\documentclass[a4paper]{article}

\usepackage[english]{babel}

\usepackage[utf8]{inputenc}
\usepackage[utf8]{inputenc}
\usepackage{amsmath}
\usepackage{graphicx}
\usepackage{amssymb}
\usepackage{amsthm}
\usepackage{tikz-cd}
\usepackage{mathrsfs}
\usepackage[colorinlistoftodos]{todonotes}
\usepackage{enumitem}
\usepackage{yfonts}
\usepackage{ dsfont }
\usepackage{MnSymbol}
\usepackage{slashed}

\title{Special Lagrangian pair of pants}

\author{Yang Li}

\date{\today}
\newtheorem{thm}{Theorem}[section]
\newtheorem{lem}[thm]{Lemma}

\theoremstyle{definition}

\newtheorem{cor}[thm]{Corollary}

\newtheorem{rmk}[thm]{Remark}
\newtheorem{prop}[thm]{Proposition}

\newtheorem*{Notation}{Notation}

\newtheorem*{Acknowledgement}{Acknowledgement}

\newcommand{\ie}{\emph{i.e.} }
\newcommand{\cf}{\emph{cf.} }

\newcommand{\R}{\mathbb{R}}
\newcommand{\C}{\mathbb{C}}
\newcommand{\Z}{\mathbb{Z}}

\newcommand{\norm}[1]{\left\lVert#1\right\rVert}

\DeclareMathOperator{\Tr}{Tr}

\def\Xint#1{\mathchoice
	{\XXint\displaystyle\textstyle{#1}}%
	{\XXint\textstyle\scriptstyle{#1}}%
	{\XXint\scriptstyle\scriptscriptstyle{#1}}%
	{\XXint\scriptscriptstyle\scriptscriptstyle{#1}}%
	\!\int}
\def\XXint#1#2#3{{\setbox0=\hbox{$#1{#2#3}{\int}$ }
		\vcenter{\hbox{$#2#3$ }}\kern-.6\wd0}}

\def\dashint{\Xint-}

\begin{document}
	\maketitle

\begin{abstract}
We construct special Lagrangian pair of pants in general dimensions, inside the cotangent bundle of $T^n$ with the Euclidean structure. 
\end{abstract}

\section{Introduction}

A familiar picture on Riemann surfaces is the \emph{pair of pants decomposition}. Each pair of pants is diffeomorphic to the thrice punctured $S^2$, and a general Riemann surface can be topologically built from gluing pair of pants along cylinders. This perspective is fruitful in many applications, such as 
the tropical degeneration of holomorphic curves \cite{Mikhalkin2}, and the cuspidal degeneration of hyperbolic metrics.

The long term goal of our project is to 
produce a large supply of \emph{special Lagrangian submanifolds} inside complex $n$-dimensional Calabi-Yau manifolds $(X,\omega,\Omega)$ with SYZ fibrations, by gluing together some (hitherto unknown) higher dimensional analogues of pair of pants, via some combinatorial pattern prescribed by a tropical hypersurface. Here special Lagrangians of phase $\hat{\theta}$ are real $n$-dimensional submanifolds with
\[
\omega|_\mathbb{L}=0,\quad \text{Im}(e^{-i\hat{\theta}}\Omega)|_\mathbb{L}=0.
\]
A fundamental observation of Harvey-Lawson \cite{HarveyLawson1} is that these are minimal surfaces. The basic challenge of the field is that the non-perturbative construction techniques are quite limited, and the simplest case of constructing new special Lagrangians inside the Euclidean $T^*T^n$ already requires new inputs.

In the case $n=2$, the local model for special Lagrangian pair of pants can be obtained through the hyperk\"ahler rotation trick (\cf (\ref{2Dsolution}) below). The main output of this paper is to construct the higher dimensional analogue for the special Lagrangian pair of pants, building upon a combination of PDE techniques developed by Caffarelli-Nirenberg-Spruck and many subsequent authors \cite{Caff}\cite{HarveyLawson}\cite{Yuan}\cite{Yuan2}, and some minimal surface theory.

It is worth pointing out that Matessi \cite{Matessi} has considered the `soft' version of our problem involving Lagrangians,  and we shall give an impressionistic review of  some of his main ideas.

\subsection{Matessi's work on Lagrangian pair of pants}

Matessi is motivated by the mirror symmetry dual construction to Mikhalkin's tropical-to-complex correspondence \cite{Mikhalkin}
. In the B-side story, one starts with a tropical hypersurface $\Gamma\subset \R^n$ with some smoothness assumptions, and aims to produce a 1-parameter family of algebraic hypersurfaces $Y_t$ in $(\C^*)^n$, whose images under the projection 
\[
\text{Log}_t: (\C^*)^n\to \R^n,\quad (z_1,\ldots z_n)\to \frac{1}{\log t} (\log |z_1|,\ldots , \log |z_n|)
\]
(known as the `amoeba')
converges to $\Gamma$ in the Hausdorff distance. In a more refined description, there is a piecewise linear lift $\hat{\Gamma}\subset (\C^*)^n$ (known as the `phase tropical hypersurface'), which has a singular fibration map to $\Gamma\subset \R^n$ with fibres being the closures of `coamoebas'.
The phase tropical hypersurface $\hat{\Gamma}$ captures the leading asymptote of $Y_t$ for $t\ll 1$, and 
as a topological manifold $\hat{\Gamma}$ is homeomorphic to $Y_t$.

Here the \emph{coamoeba} is just a subset of $T^n$
consisting of two copies of an  $n$-dimensional open simplex glued at the $(n+1)$ vertices. Let
\begin{equation}
\Delta_n= \{ 0\leq x_i\leq \pi, \forall i=1,\ldots n,\quad \sum_1^n x_i\leq \pi\}, \quad -\Delta_n=\{ x\in T^n|-x\in \Delta_n\},
\end{equation}
then the standard coamoeba is
\begin{equation}\label{coamoeba}
C_{std}= \text{Int}(\Delta_n\cup -\Delta_n)\cup \{ \text{vertices of $\Delta_n$} \}.
\end{equation}
Notice that $\Delta_n, -\Delta_n$ share the vertices $(0,\ldots 0), (\pi,0,\ldots 0), \ldots (0,\ldots 0,\pi)\in T^n$. The boundary of a coamoeba consists of lower dimensional coamoebas.

Matessi's goal is to imitate this picture in the Lagrangian setting (\ie the A-side). Given a tropical hypersurface $\Gamma$ in $\R^n$, one can construct a Lagrangian piecewise-linear (PL) lift $\hat{\Gamma}\subset (T^*T^n=(\R/2\pi\Z)^n\times \R^n, \omega=\sum dx_i\wedge dy_i)$ roughly as follows. 
The tropical hypersurface $\Gamma$ is a union of $(n-1)$-dimensional rational polyhedra glued along faces, with a balancing condition at the intersecting faces. The PL lift $\hat{\Gamma}$ is a piecewise linear object inside $T^*T^n$ projecting to $\Gamma$. Given a point $y$ inside the interior of an $(n-1)$-dimensional face of $\Gamma$, the fibre over $y$ in $\hat{\Gamma}$ is the $S^1\subset T^n$ spanned by the $\omega$-orthogonal complement direction to the $(n-1)$-dimensional face. More generally, if $y$ is a point in the interior of the $k$-dimensional face $\Gamma_J$ of $\Gamma$, then the fibre is the closure $E_J$ of an $(n-k)$-dimensional coamoeba, which lives inside the subtorus $T^{n-k}\subset T^n$ spanned by the $\omega$-orthogonal complement of the face. The PL lift is then
\[
\hat{\Gamma} = \cup_J \Gamma_J\times E_J.
\]
The $\omega$-orthogonal complement property guarantees that it is a singular Lagrangian. The correspondence between the faces of the tropical hypersurface and the coamoebas is inclusion-reversing, so that the boundary components of  $\Gamma_J\times E_J$ cancel out.

The main problem is that $\hat{\Gamma}$ is singular, and one would like to find smooth Lagrangians which approximate $\hat{\Gamma}$. Matessi \cite{Matessi} constructs the following local model, called the \emph{Lagrangian pair of pants}. Define the explicit function \footnote{Matessi used slightly different conventions for the dimension and the periodicity of the lattice.} on the interior of the coamoeba (\ref{coamoeba}), 
\[
 \begin{cases}
F(x)=(\sin(\frac{x_1}{2})\ldots \sin(\frac{x_n}{2})\sin( \frac{\pi-\sum_1^n x_i}{2}  )^{1/n}, & \quad \text{on } \text{Int} (\Delta_n),
\\
F(-x)= -F(x) , & \quad \text{on } \text{Int}(-\Delta_n),
\end{cases}
\]
and consider the graph of $dF$, which is a Lagrangian inside $T^*T^n$.
As $x$ tends to the boundary of the coamoeba, but not to one of the  $(n+1)  $ vertices, then $|dF|\to \infty$. This divergence has the important geometric meaning that the boundary of the coamoeba (except for the vertices) does not contribute to the finite distance boundary of the Lagrangian graph. To capture the limit as $x$ tends to the vertices, one performs the real blow up $\tilde{C}_{std}$ of the coamoeba $C_{std}$ at these vertices (\cf section \ref{Topologysection} below), and shows that $dF$ extends smoothly to $\tilde{C}_{std}$. This gives a 1-parameter family of smooth Lagrangian embedding of $\tilde{C}_{std}$  into $T^*T^n$ via the map $(x, \lambda dF(x))$ \cite[Lem 3.5]{Matessi}. Matessi calls the images $L_\lambda$ of this embedding the `rescaled $n$-dimensional Lagrangian pair of pants'.

In the case $n=2, 3$, Matessi \cite[Prop. 3.27]{Matessi} shows that the Lagrangian pair of pants is homeomorphic to the PL lift of the standard tropical hypersurface
\[
\Gamma_{std}= \text{non-smooth locus of } \max(0, y_1,\ldots y_n) \subset \R^n_y,
\]
and by analysing the asymptotic behaviour of $L_\lambda$, he shows that $L_\lambda$ converges as $\lambda\to 0$ in the Hausdorff topology to the PL lift $\hat{\Gamma}_{std}$ of $\Gamma_{std}$ \cite[Cor. 3.9]{Matessi}. He then proceeds to use these local models to construct Lagrangian smoothings of PL lifts of more general smooth tropical hypersurfaces, in the case of dimension $n=2, 3$, with some extensions to other toric manifolds \cite[Thm 1.1]{Matessi}. The same results are conjectured to hold in all higher dimensions, but the asymptotic behaviour near the boundary of the coamoeba will be combinatorially more complex.

\begin{rmk}
Matessi pointed out to the author that his subsequent paper \cite[section 9.4]{Matessi2} has speculated on the construction of special Lagrangian pair of pants, through an inductive solution of Dirichlet problems. The interplay between Lagrangian submanifolds and tropical geometry is a popular research topic, \cf \cite[Introduction]{Matessi2} for many other works.
\end{rmk}

\subsection{Special Lagrangian pair of pants}

We now give a summary of the main results. The ambient space is $T^*T^n$, with the Euclidean structure
\[
\begin{cases}
\omega= \sum dx_i\wedge dy_i,
\\
\Omega= \bigwedge_i (dx_i+ \sqrt{-1} g_{ij} dy_j),
\end{cases}
\]
where $x_i\in \R/2\pi \Z$ are the angle coordinates, and $y_i\in \R$ are the moment maps, and $(g_{ij})$ is an $n\times n $ positive definite symmetric matrix.

\begin{thm}
Let $n\geq 2$.
There exists a special Lagrangian $\mathbb{L}_n$ of phase $\hat{\theta}= \frac{n-1}{2}\pi$ inside $T^*T^n$, such that 
\begin{enumerate}
\item (Smoothness) 
$\mathbb{L}_n$ is an embedded submanifold of $T^*T^n$ without boundary (\cf Prop. \ref{integralcurrentnoboundary}, \ref{smoothness}).

    \item (Coamoeba)  Under the projection $T^*T^n\to T^n$, the image of $\mathbb{L}_n$ is the standard coamoeba $C_{std}$ (\cf Prop. \ref{supportL}).

    \item (Fibres) Over $\text{Int}(\Delta_n\cup -\Delta_n)$, the special Lagrangian $\mathbb{L}_n$ can be written as the graph of the gradient of some potential function $u_n$. Over each of the $(n+1)$-vertices, the preimage is an $(n-1)$-dimensional real analytic hypersurface in the contangent fibre, which can be realised as the boundary of a convex set in $\R^n$  (\cf Prop. \ref{supportL}, Cor. \ref{graphicalnearvertex}).

     \item (Topology) The projection $\mathbb{L}_n\to C_{std}$ lifts to a real analytic map to the real blow up $\mathbb{L}_n\to \tilde{C}_{std}$, which is a homeomorphism (\cf Prop. \ref{Topology}).

    \item (Tropical hypersurface) Under the projection $T^*T^n\to \R^n$, the image of $\mathbb{L}_n$ lies within bounded distance to the standard tropical hypersurface $\Gamma_{std}\subset \R^n_y$ (\cf Cor. \ref{tropicalhypersurface}).
\end{enumerate}

\end{thm}

Morever the construction has an inductive pattern. The $n=1$ case is just the zero section $S^1\subset T^*S^1$.  The $n=2$ case is the pair of pants known from the hyperk\"ahler rotation trick.

\begin{thm}\label{asymptotic}(\cf section \ref{asymptoticproof})
Let $n\geq 2$. The special Lagrangian submanifold $\mathbb{L}_n$ satisfies that
\begin{enumerate}
    \item (Regularity)  Around any point on $\mathbb{L}_n$, there is an ambient ball of radius bounded below by a constant depending only on $n,g$, such that within the ball $\mathbb{L}_n$ is graphical over an $n$-plane, with $C^{k,\alpha}$ norm bounded by $C(n,g, k,\alpha)$. In particular, the $C^{k,\alpha}$-topology is well defined.

    \item (Inductive asymptote)  In the region where $y_n$  is sufficiently negative depending on $n, g$, the special Lagrangian $\mathbb{L}_n$ is the graph of a $C^{k,\alpha}$-small normal vector field over $\mathbb{L}_{n-1}\times \R_{y_n} $, and the local $C^{k,\alpha}$-norm 
    satisfies the exponential decay bound $O(e^{ y_n/C(n,g)})$ as $y_n$ tends to $-\infty$.  Similar asymptotes hold when one of $y_1,\ldots y_{n-1}, -\sum_1^n y_i$ is sufficiently negative. These regions cover the complement of a compact subset in $T^*T^n$.
\end{enumerate}

\end{thm}

The basic strategy is to inductively solve a sequence of Dirichlet problems depending on $n$, for the special Lagrangian graph equation over the $n$-dimensional simplex $\Delta_n$. In each step, the lower dimensional solutions are utilised to prescribe the boundary data of the $n$-dimensional problem. The main work is to prove the existence and the properties of these PDE solutions $u_n$, and interpret the results geometrically.

From the PDE perspective the principal novelty is the behaviour of the solution near the domain boundary. While the regularity theory for the convex solutions to the special Lagrangian graph equation is well established, the known global regularity results almost inevitably relies on some \emph{strict (pseudo)convexity} assumption on the boundary of the domain, so that one can build some barrier function out of the boundary defining function, to achieve boundary Lipschitz estimates and higher order derivative estimates \cite{Caff}\cite{HarveyLawson}. In contrast, the domain in our setting is the simplex $\Delta_n$, where the strict boundary convexity fails, so both the existence and the properties  of solutions require new types of barrier constructions.

In the transition from PDE to geometry, an important step is to show that \emph{the gradient $du_n(x)$ diverges to infinity} as $x$ tends to any point of $\partial\Delta_n$ except for the  $(n+1)$ vertices (\cf Cor. \ref{gradientdivergence}). This behaviour is opposite to what happens in the standard theory for smooth and strictly convex domains, and the proof involves a delicate barrier argument. To obtain the special Lagrangian $\mathbb{L}_n$, we take a reflected copy of the solution over $-\Delta_n$, and take the \emph{closure} of the special Lagrangian graph over $\text{Int}(\Delta_n\cup (-\Delta_n))$. The gradient divergence then has the geometric significance that $\partial \Delta_n$ (except for the vertices) has no contribution to the boundary of the special Lagrangian as an integral current. However, the $(n+1)$ vertices shared by $\Delta_n$ and $-\Delta_n$ bring in extra points in the step of taking the closure, and the happy fact is that their contributions to the current boundary $\partial \mathbb{L}_n$ \emph{cancel} exactly by an involution symmetry.


The asymptotic information of the special Lagrangian  corresponds to the boundary behaviour of $u_n$ in the PDE perspective.  Our boundary behaviour is complicated from the original PDE viewpoint: for instance, an infinite amount of volume of the special Lagrangian is concentrated near $\partial \Delta_n$. This complexity is partly caused by describing the special Lagrangian using the domain coordinates on $\Delta_n$, and the geometric perspective affords the flexibility to use alternative coordinate systems. Away from the vertices, each boundary face is associated with a new coordinate chart mixing the position and momentum variables, which analytically corresponds to taking \emph{partial Legendre transform}, with the effect that the special Lagrangian $\mathbb{L}_n$ is a small graph over the new coordinates. The behaviour near the vertices is more delicate, and to prove the \emph{smoothness} of $\mathbb{L}_n$ and the \emph{asymptotic exponential decay} estimates, we make use of minimal surface theory.

\begin{Notation}
In this paper $C\geq 1$ will stand for constants which depend only on the dimension and the ambient metric on $T^*T^n$, and do not depend on the other auxiliary parameters, unless explicitly indicated otherwise. The notation $a\lesssim b$ means $a\leq Cb$ for some estimable constant $C$.
Since this paper will involve many auxiliary barrier constructions depending on multiple parameters, the notations involved in each barrier will be localised to the subsection, and then recycled in later subsections. The $C^{k,\alpha}$-norms will always refer to bounded balls inside local charts around a given point, rather than taking the global supremum over the special Lagrangian.

\end{Notation}

\begin{Acknowledgement}
The author is a current Clay Maths Research Fellow, based at MIT. He thanks D. Matessi for comments, and Y-S. Lin, S. Esfahani, and I. Zharkov for discussions on related topics.
\end{Acknowledgement}

\section{Dirichlet problem}

\subsection{Special Lagrangian graph equation}\label{specialLaggrapheqnsection}

We look for special Lagrangians inside the the cotangent bundle of $T^n$ with the following Euclidean structure:
\[
\begin{cases}
\omega= \sum dx_i\wedge dy_i,
\\
\Omega= \bigwedge_i (dx_i+ \sqrt{-1} g_{ij} dy_j),
\end{cases}
\]
where $x_i\in \R/2\pi \Z$ are the angle coordinates, and $y_i\in \R$ are the moment maps, and $(g_{ij})$ is an $n\times n $ positive definite symmetric matrix. In particular, the $T^n$ fibres are special Lagrangians of phase zero.

We consider special Lagrangians which are graphical, namely $y_i$ can be expressed as the gradient of a potential
\[
y_i= \frac{\partial u}{\partial x_i}. 
\]
The special Lagrangian condition with phase $\hat{\theta}$ means that $\text{Im} (e^{-\sqrt{-1}\hat{\theta} }\Omega)=0$, and translates into the second order equation
\begin{equation}\label{specialLaggraph1}
\text{Im} (e^{-\sqrt{-1}\hat{\theta} }\det \left(  I+ \sqrt{-1} g_{ik} \frac{\partial^2 u}{\partial x_k\partial x_j} \right) ) =0.
\end{equation}

The case where $g_{ij}=\delta_{ij}$ is the familiar \emph{special Lagrangian graph equation}
\[
\text{Im} (e^{-\sqrt{-1}\hat{\theta} }\det \left(  I+ \sqrt{-1} D^2 u \right) ) =0.
\]
In fact we can reduce the equation to the standard form by a linear change of coordinates. To see this, let $(a_{ij})\in GL_+(n, \R) $ with inverse matrix $(a^{ij})$, and
\[
x_i= a_{ij} \tilde{x}_i, \quad y_i= a^{ji} \tilde{y}_j,
\]
so that $\omega= \sum d\tilde{x}_i\wedge d\tilde{y}_i$, and $\tilde{y}_i= \frac{\partial u}{\partial \tilde{x}_i} $.
We find
\[
\Omega= \det(a) \bigwedge_i (d\tilde{x}_i+ \sqrt{-1} \tilde{g}_{ij} d\tilde{y}_j),\quad \tilde{g}_{lk}= a^{li}g_{ij}a^{kj}.
\]
We can find the matrix $(a_{ij})$ so that $\tilde{g}_{ij}=\delta_{ij}$, since all positive definite quadratic forms are congruent to the standard form. The special Lagrangian graph equation in the $\tilde{x}_i$ coordinates is then in the standard form.

It is well known that the special Lagrangian consists of phase branches. 
In the coordinate independent language, the Hessian of $u$ is a symmetric quadratic form, and we can define an operator $A$ by 
\[
D^2 u (v_1, v_2)= g(Av_1, v_2).
\]
Since $A$ is self-adjoint with respect to  the inner product $g=g^{ij} dx_idx_j$, its eigenvalues are real:
\[
\lambda_1\leq \lambda_2\leq \ldots \leq \lambda_n,
\]
and the special Lagrangian graph equation can be written as
\[
\arctan D^2 u:=\sum_1^n \arctan \lambda_i= \hat{\theta} \mod \pi \Z.
\]

The analytic behaviour of the equation
\begin{equation}\label{specialLaggraph2}
\arctan D^2  u= \hat{\theta}
\end{equation}
depends strongly on the range of phase angle $\hat{\theta}$. The range of arctan forces 
 $\hat{\theta}\in (-\frac{n\pi}{2}, \frac{n\pi}{2})$, and up to replacing $u$ by $-u$ we only need to consider $\hat{\theta}\geq 0$. A very incomplete list of notable results include:
\begin{itemize}
    \item  In the range $\hat{\theta}\in [\frac{n-1}{2}\pi, \frac{n}{2}\pi)$, so that $u$ is convex, Caffarelli-Nirenberg-Spruck \cite{Caff} proved the existence and uniqueness of solution for smooth boundary data, over smooth and strictly convex bounded domains in $\R^n$.

    \item In the `supercritical' range $\hat{\theta} \in [\frac{n-2}{2}\pi, \frac{n}{2}\pi)$, Wang-Yuan \cite{Yuan2} proved the a priori interior Hessian estimate 
    \[
    \max_{B_R} |D^2 u| \leq C(n, \text{osc}_{B_{2R}}( \frac{u}{R} ) ),
    \]
which implies the interior real analyticity of viscosity solutions. The convex case was previously settled in \cite{Yuan}. 

\item For all phases $\hat{\theta}\in (-\frac{n}{2}\pi, \frac{n}{2}\pi)$, Harvey-Lawson \cite{HarveyLawson} proved the existence and uniqueness of $C^0$ viscosity solution for any continuous boundary data, over smooth bounded domains whose boundary is pseudoconvex in an appropriate sense.

\item 
When $|\hat{\theta}|< \frac{n-2}{2}\pi$ the regularity property of the viscosity solution can be very bad. A very recent sample result by Mooney-Savin \cite{MooneySavin} constructed Lipschitz viscosity solutions in 3D which fails to be $C^1$.

\end{itemize}

We will be interested in solving the equation over the simplex, which is neither smooth nor \emph{strictly} convex, so some care is needed in applying the known existence results.

\subsection{The dimension two case}

We now consider the special case $n=2$, with the special choice of phase angle $\hat{\theta}= \frac{\pi}{2}$. We shall use the hyperk\"ahler rotation trick to find the special solution. Notice that
\[
\sqrt{ g_{11} g_{22}-g_{12}^2 } \omega, \text{Re}{\Omega}, \text{Im} \Omega
\]
form a hyperk\"ahler triple. The special Lagrangian condition is equivalent to being holomorphic with respect to the complex structure defined by the (2,0)-form
$\sqrt{ g_{11} g_{22}-g_{12}^2 } \omega+ \sqrt{-1} \text{Re} \Omega $. We find holomorphic coordinates with respect to this new complex structure
\[
\begin{cases}
z_1= \exp(   \sqrt{ g_{11} g_{22}-g_{12}^2 } y_2-\sqrt{-1} x_1),
\\
z_2= \exp(   \sqrt{ g_{11} g_{22}-g_{12}^2 } y_1+\sqrt{-1} x_2),
\end{cases}
\]
so that $T^*T^2$ can be identified with $(\C^*)^2_{z_1, z_2}$. The standard pair of pants surface is
\[
\{  z_1+ z_2=1\} \subset (\C^*)^2.
\]
Solving  $y_1, y_2$ in terms of $x_1, x_2$ from the equation
\[
\begin{cases}
\exp(   \sqrt{ g_{11} g_{22}-g_{12}^2 } y_2) e^{\sqrt{-1} x_1} +  \exp(   \sqrt{ g_{11} g_{22}-g_{12}^2 } y_1) e^{-\sqrt{-1} x_2}=1,
\\
\exp(   \sqrt{ g_{11} g_{22}-g_{12}^2 } y_2) e^{-\sqrt{-1} x_1} +  \exp(   \sqrt{ g_{11} g_{22}-g_{12}^2 } y_1) e^{\sqrt{-1} x_2}=1,
\end{cases}
\]
we obtain
\[
y_1=  \frac{1}{  \sqrt{ g_{11} g_{22}-g_{12}^2 }   } \log \frac{\sin x_1}{ \sin (x_1+x_2)}  ,\quad y_2=  \frac{1}{  \sqrt{ g_{11} g_{22}-g_{12}^2 }   }\log \frac{\sin x_2}{ \sin (x_1+x_2)}.
\]
Notice the $x$-domain is contained in the closure of the coamoeba $\Delta_2\cup -\Delta_2\subset T^2$,
\[
 \Delta_2= \{ 0\leq x_1 \leq \pi, 0\leq x_2\leq \pi, 0\leq x_1+x_2\leq \pi  \}.
\]

We can recover the potential $u$  from $du=y_1dx_1+y_2dx_2$. 
 We introduce the special function
\begin{equation}\label{specialfunctionL}
   L(x) = \int_0^x \log \sin t dt,\quad 0\leq x\leq \pi.
\end{equation}
Then the  solution $u$ over $\Delta_2$ is
\begin{equation}\label{2Dsolution}
u_2= \frac{1}{  \sqrt{ g_{11} g_{22}-g_{12}^2 }   } ( L(x_1)+ L(x_2)+ L(\pi-x_1-x_2)-L(\pi)),
\end{equation}
and the solution over $-\Delta_2$ is given by $u_2(-x)= -u_2(x)$. The three edges of the triangle $\Delta_2$ appear on symmetric footing as expected. The notation $u_2$ is meant to indicate the 2D problem.

\begin{rmk}
The pair of pants surfaces come in families $\{ az_1+bz_2=1\} $ for $a, b\in \C^*$. The effect of $a, b$ is to translate the surface inside $T^*T^2$. In particular, if $a, b\in U(1)$, the $x$-domain is translated inside $T^2$, and if $a, b>0$, then the $y_i$ are translated in $\R^2$, corresponding to the more general solutions over $\Delta_2$,
\[
u=u_0+ a_0+ a_1 x_1+ a_2 x_2,
\]
for real constants $a_0, a_1, a_2$. 
\end{rmk}

\begin{rmk}
If we start with other choices of phase angle $\hat{\theta}\in [0, \pi)$, the solution will no longer be graphical over the $x_i$ variables. Instead $x_i$ and $y_i$ will be mixed up, and we do not see an obvious path for higher dimensional generalisations.
\end{rmk}

We now seek a better analytic understanding of the special solution (\ref{2Dsolution}), which will guide our higher dimensional generalisation. An easy calculus exercise gives

\begin{lem}\label{specialfunctionLproperty}
\[
L'(x)= \log \sin x, \quad L''(x)=  \cot x, \quad L(x)+L(\pi-x)= L(\pi). 
\]
In particular $L$ is a negative, decreasing function, is convex for $0\leq x\leq \frac{\pi}{2}$, and $L(x)= x\log x+O(x)$ for $x\to 0$.
\end{lem}

\begin{lem}\label{2DHessian}
The function $\sqrt{ g_{11} g_{22}-g_{12}^2 } u_2$ can be rewritten as $L(x_1)+L(x_2)-L(x_1+x_2)$, and its Hessian matrix is explicitly calculated as
\[
\begin{pmatrix}
\cot x_1-\cot (x_1+x_2) & -\cot (x_1+x_2) \\
-\cot (x_1+x_2) & \cot x_2 -\cot (x_1+x_2)
\end{pmatrix}.
\]
This matrix is positive definite, and has determinant one. (Hint: use the trigonometric formula $\tan(x+y)= \frac{\tan(x)+\tan(y)}{1-\tan x\tan y}$.)
\end{lem}

\begin{rmk}
Using the formula $\cot x= x^{-1}- \frac{x}{3}+O(x^3)$, we see that for $(x_1,x_2)$ close to the origin, the Hessian matrix is to leading order
\[
\frac{1}{x_1+x_2} ( \frac{x_2}{x_1} dx_1^2 -2 dx_1dx_2+ \frac{x_1}{x_2} dx_2^2) +O(x_1+x_2). 
\]
It has a large eigenvalue comparable to $\min(x_1, x_2)^{-1}$, and a small eigenvalue comparable to $\min(x_1, x_2)$, with product one.


\end{rmk}

\begin{cor}
The solution $u_2$ has boundary data zero on $\partial \Delta_2$, is convex on $\Delta_2$, and is a smooth negative valued function in the interior. Near the edge $x_2=0$ but away from the vertices of $\Delta_2$, the boundary behaviour is
\[
\begin{cases}
u_2=  \frac{1}{  \sqrt{ g_{11} g_{22}-g_{12}^2 }   } x_1\log x_1 +O(x_1), \\
\frac{\partial u_2}{\partial x_1}= \frac{1}{  \sqrt{ g_{11} g_{22}-g_{12}^2 }   } \log x_1 +O(1),   
\\
\frac{\partial^2 u_2}{\partial^2 x_1}= \frac{1}{  \sqrt{ g_{11} g_{22}-g_{12}^2 }   }  \frac{1}{x_1} +O(1),
\end{cases}
\]
and similarly for the other two edges. Near the vertices, the gradient of $u_2$  converge to finite values along the rays $\frac{x_1}{x_2}=const>0$, but the largest eigenvalue of the Hessian of $u_2$ tends to $+\infty$. 
\end{cor}

\begin{rmk}
The zero boundary data is related to the observation that $u_1=0$ solves the 1D version of the special Lagrangian graph equation. The fact that the Hessian becomes degenerate along the edge can be expected as follows: the zero boundary data on the edge suggests that the smallest eigenvalue of the Hessian should tend to zero on the boundary, and the sum of arctan of the eigenvalues is $\frac{\pi}{2}$, so the arctan of the largest eigenvalue should tend to $\frac{\pi}{2}$.
\end{rmk}

\subsection{Inductive boundary value prescription}

We now propose an inductive strategy to produce special Lagrangian graphs in general dimensions. The main guiding principle is that \emph{the $n$-dimensional solution should supply the boundary value for the $(n+1)$-dimensional Dirichlet problem}. The domain in the dimension $n$ case is 
\begin{equation*}
\Delta_n= \{ 0\leq x_i \leq \pi, i=1,2,\ldots , n, \text{and } 0\leq x_1+x_2+\ldots +x_n\leq \pi  \}.
\end{equation*}

For concreteness, we will first describe the Dirichlet problem in 3D. The domain is the 3-dimensional simplex $\Delta_3$.
The boundary $\partial \Delta_3$ consists of 4 triangles. On all the 6 edges of the 4 triangles, we put the boundary data $u_1=0$. On each of the 4 triangle faces, we use the solution to the 2D version of the problem to prescribe the boundary data for $u$:
\[
\begin{cases}
u(x_1, x_2,0)= \frac{1}{  \sqrt{ g_{11} g_{22}-g_{12}^2 }   } ( L(x_1)+ L(x_2)+ L(\pi-x_1-x_2)-L(\pi)),
\\
u(x_1, 0, x_3)= \frac{1}{  \sqrt{ g_{11} g_{33}-g_{13}^2 }   } ( L(x_1)+ L(x_3)+ L(\pi-x_1-x_3)-L(\pi)),
\\
u(0, x_2,x_3)= \frac{1}{  \sqrt{ g_{22} g_{33}-g_{23}^2 }   } ( L(x_2)+ L(x_3)+ L(\pi-x_2-x_3)-L(\pi)),
\end{cases}
\]
and similarly for the face $\{ x_1+x_2+x_3=\pi\} $. By construction, the boundary data match continuously on the intersection of faces. The Dirichlet problem in 3D is to find $u=u_3$ with the above boundary data on $\partial \Delta_3$, solving the special Lagrangian graph equation (\ref{specialLaggraph1})(\ref{specialLaggraph2}) in the phase branch $\hat{\theta}= \frac{(3-1)}{2}\pi$, such that $u_3$ is smooth in the interior of $\Delta_3$, and continuous up to the boundary.

More generally, assuming the $(n-1)$-dim problem has a solution, then the $n$-dim problem has well defined boundary data, which is $C^0$ on $\partial \Delta_n$  by construction, and we seek a solution $u_n$ to (\ref{specialLaggraph1})(\ref{specialLaggraph2}) with $\hat{\theta}= \frac{(n-1)}{2}\pi$, such that $u_n$ is smooth in the interior of $\Delta_n$, and continuous up to the boundary.

\begin{rmk}
We emphasize that the phase branch $\hat{\theta}= \frac{(n-1)}{2}\pi$ automatically implies the convexity of the solution. This easily implies by induction that the Dirichlet solution is nonpositive.
\end{rmk}


\subsection{Existence theorem and  barrier functions}

We shall prove 

\begin{thm}\label{Existence}
(Existence) Let $n\geq 1$. There is a unique solution $u_n$ to the Dirichlet problem in dimension $n$, which is smooth in the interior of $\Delta_n$, and continuous up to the boundary of $\Delta_n$. More precisely, we have the following boundary modulus of continuity estimate: 
\begin{equation}\label{boundarymodulus}
   |u_n(x_1,\ldots x_n) - u_{n-1}(x_1, \ldots x_{n-1})| \leq  - C(n) x_n \log \frac{x_n}{2\pi} ,\quad \forall x\in \Delta_n, 
\end{equation}
for some constant $C(n)$ depending only on the dimension $n$ and $g_{ij}$.
Similarly with the other boundary components. 
\end{thm}

\begin{rmk}
Notice that $0\leq x_n\leq \pi$, so the log term on the right is negative. For $x_n$ bounded from below, the estimate (\ref{boundarymodulus}) amounts to a lower bound on $u_n$. The main force of (\ref{boundarymodulus}) is when $x_n$ is small. By induction (\ref{boundarymodulus}) implies
\begin{equation}\label{boundedmodulus2}
|u_n(x_1,\ldots x_n)-u_m(x_1,\ldots x_m)|\leq -C(n)\sum_{m+1}^n x_i\log \frac{x_i}{2\pi},\quad 1\leq m\leq n-1.
\end{equation}
We may call this $O(x\log x)$ boundary continuity near the face $\Delta_m\subset \partial \Delta_n$. In the statement we focus on the $\{x_n=0\} $ boundary, but of course the same kind of boundary modulus of continuity applies to any other boundary faces.
\end{rmk}

The case of $n=1,2$ have been checked explicitly, so we will focus on $n\geq 3$, and assume by induction that the cases $\leq n-1$ are known. The key is to build barrier functions.

We begin with the upper barrier construction. 

\begin{lem}
The function \[
\bar{u}_n =(1- \frac{x_n}{\pi}  ) u_{n-1} (  \frac{x_1}{ 1- \frac{x_n}{\pi}   }, \ldots \frac{x_{n-1} }{ 1- \frac{x_n}{\pi}   } ),
\] 
is a supersolution to the $n$-dim Dirichlet problem. 
\end{lem}

\begin{proof}
For our choice of phase, the solutions to the $m$-dimensional problems are automatically convex for $m=1,2,\ldots n-1$. These lower dimensional solutions prescribe the boundary data of the $n$-dim problem. By construction $\bar{u}_n$ is linear on the line segments joining $(0,\ldots 0,\pi)$ and the points of $\Delta_{n-1}\subset \partial \Delta_n$, and agree with the boundary data $u_n|_{\partial \Delta_n}$ at the endpoints of these line segments. Hence $\bar{u}_n \geq u_n|_{\partial \Delta_n}$ on $\partial \Delta_n$.


The induction hypothesis implies $\bar{u}_n$ is smooth in the interior of $\Delta_n$ and continuous up to the boundary. Since $\bar{u}_n$ is linear when restricted to the line segments connecting $(0,\ldots 0, \pi)$ to $(x_1,\ldots x_{n-1}, 0)$, its smallest Hessian eigenvalue cannot be positive. Thus
\[
\arctan D^2 \bar{u}_n \leq  \frac{n-1}{2} \pi 
\]
on the interior of $\Delta_n$, namely $\bar{u}_n$ is a \emph{supersolution} to the Dirichlet problem. 
\end{proof}

Our next goal is to build lower barrier functions.
We shall first describe the simpler case, when the dihedral angle between any two codimension one faces of $\Delta_n$ is acute.

The barrier function in this `acute simplex' case is
\[
\underbar{u}_n= u_{n-1}(\tilde{x}_1,\ldots \tilde{x}_{n-1}) +\sum_{i=1}^{n-1} ( L(x_n)+ L(x_i)-L(x_i+x_n) ) 
+  K x_n \log \frac{x_n}{2\pi}
\]
Here $\tilde{x}_i$ are linear functions of the form $x_i+ a_i x_n$ for some $a_i\in \R$, such that $g(d\tilde{x}_i, dx_n)=0$. The advantage of the `acute simplex' assumption is that $(\tilde{x}_i)\in \Delta_{n-1}$, so that $u_{n-1}(\tilde{x}_i) $ makes sense. The parameter $K\gg 1 $ is a large positive constants depending on $n$ to be determined. Notice that $\underbar{u}_n $
agrees with $u_{n-1}$ on the $x_n=0$ boundary face.

\begin{lem}
The function $\underbar{u}_n$ is a subsolution to the $n$-dim Dirichlet problem.
\end{lem}

\begin{proof}
For $K\gg C(n-1)$, the negativity of the term  $K x_n \log \frac{x_n}{2\pi}$ for $x_n>0$ would overwhelm the boundary modulus of continuity in the lower dimensional problems, hence on $\partial \Delta_n$, we have $\underbar{u}_n\leq u_n|_{\partial \Delta_n}$.

We now estimate the Hessian eigenvalues of $\underbar{u}_n$. Notice that the Hessian of $   L(x_n)+ L(x_i)-L(x_i+x_n)$ (computed in Lem \ref{2DHessian}) is bounded below by $C^{-1} x_n dx_i^2- \frac{C}{x_n} dx_n^2$. Thus after summation,
\[
D^2 \underbar{u}_n \geq D^2 u_{n-1} +C^{-1} x_n\sum_1^{n-1} dx_i^2 + (K-C) x_n^{-1}dx_n^2,
\]
where we freely change $C$ from line to line. Let $e_1, \ldots e_{n-1}$ be the orthonormal eigenbasis of $D^2u_{n-1}$  on $\R^{n-1} $, with respect to the metric $g=g^{ij} dx_idx_j$, and denote the eigenvalues as $\lambda_1'\leq \ldots \leq\lambda_{n-1}'$. By induction $\sum_1^{n-1} \arctan \lambda_i'=\frac{n-2}{2}\pi$, so $0\leq \lambda_1' \leq  \cot \frac{\pi}{2n}\ll K$.

We  require $K\gg C$. Then
\[
C^{-1} x_n\sum_1^{n-1} dx_i^2 + (K-C) x_n^{-1}dx_n^2\geq  \frac{1}{2C} x_n\sum_1^{n-1} d\tilde{x}_i^2 + \frac{K}{2} x_n^{-1}dx_n^2
\]
so we get a lower bound (with some changed constant $C$),
\[
D^2 \underbar{u}_n \geq \sum_1^{n-1} (\lambda_i' +C^{-1}x_n) e_i^*\otimes e_i^* + \frac{K}{2x_n} dx_n^2. 
\]
The main point of using $\tilde{x}_i$ in favour of $x_i$ is that $e_i^*, dx_n$ are all $g$-orthogonal, so we can easily estimate the matrix arctan of the RHS:
\[
\begin{split}
& \arctan D^2 \underbar{u}_n\geq \sum_1^{n-1} \arctan (\lambda_i' +C^{-1}x_n) + \arctan \frac{K}{Cx_n} 
\\
& \geq \sum_1^{n-1} \arctan \lambda_i' +C^{-1}x_n + ( \frac{\pi}{2}- CK^{-1}x_n ) 
\\
& = \frac{(n-1)}{2} \pi+ C^{-1}x_n- CK^{-1}x_n.  
\end{split}
\]
Here the second line uses that $\lambda_1'\leq \cot (\frac{\pi}{2n})\leq C$. Choosing $K\gg C$, we obtain
\[
\arctan D^2 \underbar{u}_n\geq \frac{(n-1)}{2} \pi,
\]
hence the claim.
\end{proof}

In the general case, without the acute simplex assumption, the above barrier is not globally defined on $\Delta_n$. Our remedy is an inductive sequence of barriers, which aims to prove boundary continuity from vertices up to codim one faces. These will depend on parameters $1\ll K_0\ll K_1\ll \ldots  K_{n-1}$.
The barrier for the vertex $(0,\ldots 0)$ is simply
\[
\underbar{u}_{n,0}= K_0 \sum_1^n x_i\log \frac{x_i}{2\pi}.
\]
The barrier for the $m$-dimensional face $\Delta_m\subset \partial\Delta_n$ is
\[
\underbar{u}_{n,m}= u_{m}(\tilde{x}_1,\ldots \tilde{x}_m) +\sum_{i\leq m}\sum_{j>m} ( L(x_i)+ L(x_j)-L(x_i+x_j) ) 
+  K_m \sum_{m+1}^n x_j \log \frac{x_j}{2\pi}.
\]
Here $\tilde{x}_i$ are linear functions of the form $x_i+ \sum_{j> m} a_{ij} x_j$ for some $a_{ij}\in \R$, such that $g(d\tilde{x}_i, dx_j)=0$ for all $i\leq m$ and $j>m$. As a caveat, $\tilde{x}_i$ depends on the choice of $m$. The barrier $\underbar{u}_{n,m}$ is only defined on the subset of $\Delta_n$ with $(\tilde{x}_i)\in \Delta_m$, namely
\[
0\leq \tilde{x}_i\leq \pi, \quad \sum_1^m \tilde{x}_i\leq \pi.
\]
Almost the same argument as above shows

\begin{lem}
On the interior of the domain of definition of $\underbar{u}_{n,m}$, it satisfies $\arctan D^2 \underbar{u}_{n,m} \geq \frac{(n-1)}{2} \pi  $. 
\end{lem}

We now prove the existence theorem \ref{Existence}. 

\begin{proof}
We can find strictly convex smooth domains $\Omega_\epsilon\subset \Delta_n$, whose boundary is within $\epsilon$-Hausdorff distance to $\partial \Delta_n$. The inductive hypothesis implies that the boundary data $u_n|_{\partial \Delta_n}$ has $O(x\log x)$-boundary modulus of continuity bound. We can extend $u_n|_{\partial \Delta_n}$ to some function on the $\epsilon$-neighbourhood of $\partial \Delta_n$ preserving the $O(x\log x)$-boundary modulus of continuity, so by restriction we obtain a function on $\partial \Omega_\epsilon$.

By the main result of Harvey-Lawson \cite{HarveyLawson}, there is a unique continuous viscosity solution $u_{n,\epsilon}$ on $\Omega_\epsilon$ with this boundary data. The induction hypothesis implies $O(x\log x)$-modulus of continuity for the barrier functions $\bar{u}_n, \underbar{u}_{n,0}$, whence on $\partial \Omega_\epsilon$, the boundary data satisfies
\[
\underbar{u}_{n,0}+ C\epsilon \log \epsilon \leq u_{n,\epsilon} \leq  \bar{u}_n - C\epsilon\log \epsilon, 
\]
Since $\bar{u}_n, \underbar{u}_{n,0}$ are supersolutions/subsolutions, the inequality holds also on $\Omega_\epsilon$. In particular, $u_{n,\epsilon}$ is bounded independent of $\epsilon$. The regularity theorem of Wang-Yuan \cite{Yuan}\cite{Yuan2} then implies uniform interior $C^{k,\alpha} $ estimates of $u_{n,\epsilon}$ on any fixed compact subdomain of $\text{Int}(\Delta_n$), independent of small enough $\epsilon$. 
We can then take a subsequential limit as $\epsilon\to 0$, to obtain a smooth limit $u_n$ on the interior of $\Delta_n$. This limit then satisfies
\[
\underbar{u}_{n,0} \leq u_{n} \leq  \bar{u}_{n} .
\]
This gives the modulus of continuity at the vertices:
\[
0\geq u_{n}\geq K_0 \sum_1^n x_i\log \frac{x_i }{2\pi}. 
\]

Next we compare $u_{n,\epsilon}$ with the lower barrier $\underbar{u}_{n,1}$ 
on the domain of definition of $\underbar{u}_{n,1}$, which may be smaller than $\Delta_n$. 
The point is that outside this domain of definition, the previous barrier $\underbar{u}_{n,0}$ already provided the requisite a priori lower bound on $u_{n,\epsilon}$.
By the choice $K_1\gg K_0$, we can ensure that on the boundary of the domain of definition, we have
\[
\underbar{u}_{n,1}+ C\epsilon\log \epsilon \leq u_{n,\epsilon}. 
\]
Passing to the $\epsilon\to 0$ limit, we  obtain the $O( \sum_2^n x_i\log x_i)$ boundary continuity along $\Delta_1$. Continuing with this inductive yoga, we obtain the $O( \sum_{m+1}^n x_i \log x_i)$ boundary continuity along $\Delta_m$, until we reach the codimension one face $m=n-1$. 
In particular, we obtain the continuity of $u_n$ up to the boundary, with the prescribed boundary value, so $u_n$ solves the Dirichlet problem.

Finally, the uniqueness of the solution follows from the standard comparison principle argument, \cf eg. Harvey-Lawson \cite[Thm 6.3]{HarveyLawson}.
\end{proof}

\subsection{Improved lower bound near vertices}

Near the vertices of $\Delta_n$ the boundary modulus of continuity bound can be improved by removing the log factor. We focus on the neighbourhood of the origin.

\begin{prop}\label{vertexcontinuity}
Let $n\geq 2$. 
For $x\in \Delta_n$, we have an improved estimate
\[
u_n(x_1,\ldots x_n) \geq 
-C(n) \sum_1^n x_i,
\]
for some constant $C(n)$ depending only on the dimension $n$ and $g_{ij}$. 
\end{prop}

The proof depends on another barrier argument. Let
\[
v= K \sum_{i,j} ( L(x_i)+L(x_j)-L(x_i+x_j) )-K\sum_1^n x_i,\quad K\gg 1.
\]

\begin{lem}
$\arctan D^2 v\geq \frac{n-1}{2}\pi$ in a neighbourhood of the origin.
\end{lem}

\begin{proof}
By symmetry, without loss $x_1\leq x_2\leq \ldots \leq  x_n\ll 1$. We will focus primarily on the case where $g_{std}=\sum dx_i^2$, and indicate how to adapt to general $g$. The Hessian of each summand $ L(x_i)+L(x_j)-L(x_i+x_j)   $ has been computed in Lem \ref{2DHessian}, which is positive definite, with a large eigenvalue comparable to $ \min(x_i,x_j)^{-1}  $
along an eigendirection $e_{i,j}$ approximated by
$- \partial_{x_i}+ \frac{x_i}{x_j} \partial_{x_j}$. Notice that $e_{1,2}, e_{2,3},\ldots , e_{n-1,n}$ are effectively linearly independent (\ie the $g_{std}$-inner product matrix for these vectors is uniformly elliptic even as $x\to 0$), so we deduce that on the $k$-dimensional subspace
\[
V_k=\text{span} \{  e_{1,2}, \ldots e_{k, k+1}     \},
\]
we have 
\[
D^2 v\geq  \frac{1}{C} K \sum_1^k x_i^{-1} e_{i,i+1}^*\otimes e_{i,i+1}^*\geq    \frac{K}{ C x_k} g_{std}.
\]
Since the general $g$ is uniformly equivalent to $g_{std}$, we see that $D^2 v\geq  \frac{K}{ C x_k} g $ on $V_k$. By the minmax characterisation of eigenvalues, the $k$-th largest eigenvalue of $D^2 v$ with respect to $g$ is bounded below by $\frac{K}{ C x_k} $, for $k=1,\ldots n-1$.

The 2D Hessian of each summand also has a small eigenvalue (with respect to $g_{std}$) comparable to $\min(x_i, x_j)$, along an eigendirection $e_{i,j}'$ approximated by $ \frac{x_i}{x_j} \partial_{ x_i} +\partial_{ x_j}$. Let $e_n= e_{n-1,n}'$, whose associated eigenvalue is bounded below by $C^{-1}x_{n-1}$. Then $e_{1,2},\ldots e_{n-1,n}, e_n$ are effectively linearly independent, and by the minmax arguments as above
\[
D^2 v \geq C^{-1} K x_{n-1} g_{std}\geq C^{-1} K x_{n-1} g.
\]
Hence the smallest eigenvalue of $D^2 v$ is bounded below by $C^{-1} K x_{n-1}$.

Combining the above,
\[
\arctan D^2 v\geq (n-1) \arctan (\frac{K}{Cx_{n-1}}) + \arctan (\frac{Kx_{n-1}}{C})
\geq \frac{n-1}{2}\pi - \frac{Cx_{n-1}}{K}+ \frac{Kx_{n-1}}{C}.
\]
Choosing $K\gg C$ then shows the Lemma.
\end{proof}

We now prove the vertex modulus of continuity Prop. \ref{vertexcontinuity}.

\begin{proof}
In the case $n=2$, without loss $x_1\leq x_2$, and we observe
\[
L(x_1)+L(x_2)-L(x_1+x_2) \geq -Cx_1\log \frac{x_2}{x_1} \geq -C(x_1+x_2).
\]
In general, we argue by induction, and assume the statement is already known for $n-1$.

Consider the barrier $v$ for $K\gg 1$, which is a subsolution for $x$ close to the origin, say $\sum x_i^2\leq \delta$. Since $v\leq 0$, this barrier function is bounded above by $-K\sum x_i$. On the locus $\sum x_i^2=\delta$, by choosing $K$ large enough, the term $-K\sum x_i$ is more negative than the $C^0$-norm of $u_n$, so the barrier lies below $u_n$ for $\sum x_i^2=\delta$. Morever, by the inductive hypothesis, by choosing $K$ large enough, the negativity of $v$ will force the barrier to lie below $u_n$ on $\{ \sum x_i^2\leq \delta \}\cap \partial \Delta_n$.

Now by the comparison principle, the barrier $v\leq u_n$ for all $x\in \Delta_n$ inside $\sum x_i^2\leq \delta$. Since $v\geq -CK\sum_1^n x_i$, 
this implies Prop. \ref{vertexcontinuity}.
\end{proof}

\subsection{Improved upper bound and gradient divergence}

Our next goal is to improve the upper bound on $u_n$ close to the faces in $\partial \Delta_n$ but away from the vertices, to deduce the divergence of the gradient. The proof involves a delicate barrier argument.
We will use the following matrix inequality, which must be well known, but we include a proof for the reader's convenience.

\begin{lem}\label{arctanmatrix}
Suppose $A$ is a positive semidefinite symmetric matrix, and $\{e_i\}$ be an orthonormal basis of $\R^n$. Then
\[
\arctan A \leq \sum_1^n \arctan A(e_i, e_i). 
\]
\end{lem}

\begin{proof}
We can write $A$ in terms of an orthonormal eigenbasis as $A=\sum_1^n \lambda_i f_i\otimes f_i$. Then using the concavity of $\arctan$ on $\R_{\geq 0}$, 
\[
\begin{split}
& \sum_1^n \arctan A(e_i, e_i)= \sum_{i} \arctan (\sum_j \lambda_j (f_j, e_i)^2 )
\\
\geq & \sum_i \sum_j (f_j, e_i)^2 \arctan \lambda_j = \sum_j \arctan \lambda_j=\arctan A.
\end{split}
\]
\end{proof}

\begin{Notation}
Given a small parameter $\delta>0$, we shall denote the shrunken boundary face
\begin{equation}\label{Deltamdelta}
\Delta_m^\delta= \{  x=(x_1,\ldots x_m)\in  \Delta_m| \text{dist}(x, \partial \Delta_m)\gtrsim \delta   \}.
\end{equation}
This notation allows for the freedom to slightly shrink the domain in later arguments.
\end{Notation}

\begin{prop}\label{logdivergencebarrier}
Let $n\geq 2$ and $1\leq m\leq n-1$. In the neighbourhood of the interior of the $\Delta_m$ face,
\[
(x_1,\ldots x_m)\in \Delta_m^\delta, \quad x_{m+1},\ldots x_n\ll 1,
\]
then
\[
\begin{split}
& u_n(x_1,\ldots x_n)-u_m(x_1,\ldots x_m) 
\\
& \leq -C(\delta)^{-1}\max_{m+1}^n x_k |\log x_k|^{1/2} + C(\delta)' \max_{m+1}^n x_k.
\end{split}
\]
where $C, C'$ are constants which depend only on $n, g, \delta$.

\end{prop}

\begin{proof}
Without loss we focus on any point $(\xi_1,\ldots \xi_m)\in \Delta_m^\delta$, and $0< \xi_{m+1}\leq \ldots \leq \xi_n\ll 1$. The goal is to estimate the value $u_n(\xi_1,\ldots \xi_n)$ from the above. We take a cutoff function $\chi(x_1,\ldots, x_n)$ supported around $(\xi_1,\ldots \xi_m,0,\ldots 0)$, which vanishes near $\partial \Delta_n\setminus \text{Int}(\Delta_m)$, so that
\[
0\leq \chi \leq 1,\quad \chi(\xi_1,\ldots \xi_n)=1,\quad |D\chi|+|D^2\chi|\leq C(\delta).
\]
We take another auxiliary one-variable function $\psi$, with the following properties:
\[
\begin{cases}
 \psi(x)=\frac{1}{\xi_n |\log \xi_n|}  x^2,\quad & x\leq 0,  
 \\
 \psi(x)>0,\quad & x>0,
\\
\psi''(x)\leq \frac{2}{\xi_n |\log \xi_n|},\quad & 0\leq x\leq \xi_n,
 \\
\psi''(x) = 0,\quad & x>\xi_n.
\end{cases}
\]
We will use the barrier function 
\[
\begin{split}
& v=\sum_{m+1}^n \psi(x_k-\xi_k) + 
(1- \frac{x_{m+1}+\ldots x_n}{\pi})u_m( \frac{x_1}{1- \frac{\sum_{m+1}^n x_k}{\pi} },\ldots \frac{x_m}{1- \frac{\sum_{m+1}^n x_k}{\pi} })
\\
& +\Lambda \chi(x_1,\ldots x_n) \sum_{m+1}^n \frac{x_k \log x_k}{  |\log \xi_n|}  +\frac{\Lambda}{2}\sum_{m+1}^n x_k,
\end{split}
\]
where  $\Lambda$ is a non-negative parameter to be determined, and we will keep track of dependence on $\Lambda$ in the estimates.

By the convexity of $u_n$, and the fact that $u_n=0$ at the vertices, 
\[
u_n(x_1,\ldots x_n)\leq (1- \frac{x_{m+1}+\ldots x_n}{\pi})u_m( \frac{x_1}{1- \frac{\sum_{m+1}^n x_k}{\pi} },\ldots \frac{x_m}{1- \frac{\sum_{m+1}^n x_k}{\pi} }),
\]
and the inequality is strict for $(x_{m+1},\ldots x_n)\in \text{Int}(\Delta_{n-m}) $.
For $\Lambda=0$, this implies $v> u_n$ on $\Delta_n$. We observe that at the point $(\xi_1,\ldots, \xi_n)$, the coefficient of the $\Lambda$ parameter is 
\[
\sum_{m+1}^n (\frac{ \xi_k \log \xi_k }{|\log \xi_n|} + \frac{1}{2} \xi_k) \leq - \frac{1}{2}\sum_{m+1}^n \xi_k<0.
\]
Thus as $\Lambda$ increases,  the value of $u_n-v$ at $(\xi_1, \ldots \xi_n)$ will eventually switch sign, so the graph of $u_n$ and $v$ must first touch at some parameter value $\Lambda_0$, where
\begin{equation}\label{Lambda0}
\begin{split}
 \frac{1}{2} \Lambda_0& \leq \xi_n^{-1} (   -u_n(\xi)+(1- \frac{\xi_{m+1}+\ldots \xi_n}{\pi})u_m( \frac{\xi_1}{1- \frac{\sum_{m+1}^n \xi_k}{\pi} },\ldots \frac{\xi_m}{1- \frac{\sum_{m+1}^n \xi_k}{\pi} }))
\\
& \leq \frac{-u_n(\xi_1,\ldots \xi_n)+ u_m(\xi_1,\ldots \xi_m)  } {\xi_n} +C(\delta)   . 
\end{split}
\end{equation}
Here the second inequality uses the interior Lipschitz bound on $u_m$.
If we know $\Lambda_0\geq |\log \xi_n|^{1/2}$, then
\[
\frac{-u_n(\xi_1,\ldots \xi_n)+ u_m(\xi_1,\ldots \xi_m)  } {\xi_n} \geq \frac{1}{2} |\log \xi_n|^{1/2}- C(\delta),
\]
which implies the conclusion of the Proposition. Thus without loss $\Lambda_0 \leq |\log \xi_n|^{1/2}$.

We will focus on the point $x=(x_1,\ldots x_n)$ where the two graphs touch. Clearly this point lies on the set $\{ \chi\sum_{m+1}^n \frac{x_k \log x_k}{  |\log \xi_n|}>0\}\subset \text{Int}(\Delta_n)$. Furthermore,
\begin{equation}\label{xk01}
\frac{\Lambda_0}{2}\sum_{m+1}^n x_k+ \sum_{m+1}^n\psi(x_k-\xi_k) 
+\Lambda_0 \chi(x_1,\ldots x_m) \sum_{m+1}^n \frac{x_k \log x_k}{  |\log \xi_n|}\leq 0.
\end{equation}
We write $x_{k_0}= \max_{m+1\leq k\leq n} x_k $, and deduce some constraints on $x_{k_0}$. If we ignore the non-negative $\psi$ terms, then
\[
\frac{1}{2} x_{k_0} \leq -\sum_{m+1}^n  \frac{x_k \log x_k}{  |\log \xi_n|} \leq (n-m) x_{k_0} |\log x_{k_0}| |\log \xi_n|^{-1} \ll 1 ,
\]
so $x_{k_0} \leq \xi_n^{\frac{1}{2(n-m)}}\ll 1$. If $x_n\leq x_{k_0}< \frac{\xi_n}{2}$, then (\ref{xk01}) implies
\[
\xi_n\leq 4\psi(x_n-\xi_n) |\log \xi_n|\leq C\Lambda_0 x_{k_0} |\log x_{k_0}|.
\]
Since we can without loss assume $\Lambda_0\leq |\log \xi_n|^{1/2}$, we obtain
\[
x_{k_0}\geq \frac{\xi_n}{ C\Lambda_0 |\log \xi_n|   } \geq \frac{\xi_n}{ C |\log \xi_n|^2   } .
\]
In summary,
\begin{equation}\label{xk0}
\frac{\xi_n}{ C |\log \xi_n|^2   }\leq x_{k_0} \leq \xi_n^{ \frac{1}{2(n-m)} }\ll 1.   
\end{equation}



Since $v\geq u_n$ and $v(x)=u_n(x)$ at the interior point $x$, the gradients of $v$ and $u_n$ agree at $x$, and $D^2 u_n(x)\leq D^2 v(x)$. In particular, $D^2v(x)\geq 0$ and
\begin{equation}\label{arctanlower}
\arctan D^2 v (x) \geq \frac{n-1}{2}\pi.
\end{equation}
We will now derive an upper bound on $\arctan D^2 v(x)$.

On the linear subspace $\text{span}\{ dx_{m+1},\ldots dx_n\}\subset T_x\R^n$, we find an eigenbasis $e_1,\ldots e_m$ for $D^2u_m$ with respect to the metric $g$, with eigenvalues $\lambda_1',\ldots \lambda_m'$. By the construction of $u_m$, we have $\sum_1^m \arctan \lambda_i'=\frac{m-1}{2}\pi$, and by the support property of $\chi$ and the interior smoothness of $u_m$, all the $\lambda_i'$ are bounded from above by constants  $C(\delta)$. The contribution of the $u_m$ term to $D^2 v(e_i, e_i)$ is bounded above by $(1+Cx_{k_0}) \lambda_i'$, and taking advantage of the fact that $D^2 v(e_i, e_i)$ only involve differentiation in the $x_1,\ldots x_m$ variables,
\[
\begin{split}
D^2 v(e_i, e_i) \leq & (1+C x_{k_0} ) \lambda_i'+ C\Lambda_0 |D^2\chi| \sum_{k=m+1}^n |x_k  \log x_k | |\log \xi_n|^{-1} 
\\
\leq & (1+C x_{k_0} ) \lambda_i'+ \frac{C(\delta)}{ |\log \xi_n|} \Lambda_0 x_{k_0} | \log x_{k_0}|
\\
\leq &  \lambda_i'+ C(\delta) (\Lambda_0+1) x_{k_0} .
\end{split}
\]
The last inequality here uses the lower bound part of (\ref{xk0}).
By the mean value inequality for $\arctan$,
\[
\arctan D^2 v(e_i, e_i) \leq  \arctan \lambda_i'+ C(\delta) (\Lambda_0+1) x_{k_0},
\]
so after summation,
\begin{equation}\label{arctan1tom}
  \sum_1^m \arctan D^2v(e_i, e_i) \leq  \frac{m-1}{2}\pi+ C(\delta) (\Lambda_0+1) x_{k_0}.
\end{equation}
We now pick $e_{m+1}$ to be the unit vector orthogonal to this $m$-dimensional subspace, satisfying 
\[
dx_i(e_{m+1})=0,\quad \forall i\in \{ m+1,\ldots n\}\setminus \{ k_0\},
\]
and we complete $e_1,\ldots e_{m+1}$ into an orthonormal basis $e_1,\ldots e_n$. We will estimate $D^2 v(e_{m+1}, e_{m+1})$. The $u_m$ term contribution is bounded by $C(\delta)$ by the support property of $\chi$. The $\psi$ term contribution is bounded by
$\frac{C}{\xi_n |\log \xi_n|}$ if $x_{k_0}\leq 2\xi_n$, and vanishes otherwise. Thus for  $x_{k_0}\leq 2\xi_n$,
\[
\begin{split}
& D^2v(e_{m+1}, e_{m+1}) \leq \frac{C}{\xi_n |\log \xi_n|} + C(\delta) + \frac{ C\Lambda_0 }{ |\log \xi_n| x_{k_0}} + \frac{ C\Lambda_0 |D\chi| |\log x_{k_0} |}{ |\log \xi_n| }
\\
& 
\leq   C(\delta) (\frac{ \Lambda_0+1 }{ |\log \xi_n| x_{k_0} }+1).
\end{split}
\]
The same estimate holds for $x_{k_0}>2\xi_n$, where the $\frac{C}{\xi_n |\log \xi_n|}$ term does not appear. Hence in both cases,
\[
\arctan D^2v(e_{m+1}, e_{m+1}) \leq \frac{\pi}{2}- C(\delta)^{-1} \min( 1, \frac{ |\log \xi_n| x_{k_0} }{ \Lambda_0+1 }).
\]
Combined with Lemma \ref{arctanmatrix} and (\ref{arctan1tom}), we obtain the upper bound
\[
\begin{split}
& \arctan D^2 v(x) \leq \sum_1^n \arctan D^2v(e_i, e_i) 
\\
& \leq \frac{n-1}{2}\pi +C(\delta) (\Lambda_0+1) x_{k_0}- C(\delta)^{-1} \min( 1, \frac{ |\log \xi_n| x_{k_0} }{ \Lambda_0+1 }).
\end{split}
\]
Comparing this with the lower bound (\ref{arctanlower}), 
\begin{equation}\label{arctanupper}
C(\delta) (\Lambda_0+1) x_{k_0}- \min( 1, \frac{ |\log \xi_n| x_{k_0} }{ \Lambda_0+1 })\geq 0.
\end{equation}
Recall that without loss $\Lambda_0\leq |\log \xi_n|^{1/2}$, so by (\ref{xk0}), 
\[
C(\delta) (\Lambda_0+1) x_{k_0} \leq C(\delta) |\log \xi_n|^{1/2} \xi_n^{\frac{1}{2(n-m)}} \ll 1,
\]
whenever $\xi_n$ is small enough depending on $C(\delta)$. Thus (\ref{arctanupper}) reduces to
\[
C(\delta) (\Lambda_0+1) x_{k_0}- \frac{ |\log \xi_n| x_{k_0} }{ \Lambda_0+1 }\geq 0,
\]
hence 
\[
\Lambda_0+1 \geq C(\delta)^{-1} |\log \xi_n|^{1/2} \gg 1,
\]
which implies the Prop. using (\ref{Lambda0}).
\end{proof}

Prop. \ref{logdivergencebarrier} has the important consequence on the divergence of the gradient. 

\begin{cor}\label{gradientdivergence}
(Gradient divergence) Let $n\geq 2$ and $1\leq m\leq n-1$.
In the neighbourhood of the interior of the $\Delta_m$ face,
\[
(x_1,\ldots x_m)\in \Delta_m^\delta,\quad x_{m+1},\ldots x_n\ll 1,
\]
then
\[
\partial_i u_n(x_1,\ldots x_n) \leq -C(\delta)^{-1} |\log  \max_{m+1}^n x_k|^{1/2} +C(\delta)',\quad i=m+1,\ldots n.
\]
 In particular, the gradient $du_n(x)$ \emph{diverges uniformly to infinity} as $x$ tends to $\partial \Delta_n$, bounded away from the vertices.

\end{cor}

\begin{proof}
Without loss $x_{m+1}\leq \ldots \leq x_n$. 
Let $K\gg 1$ be a parameter much larger than the $C(\delta)$ from Prop. \ref{logdivergencebarrier}, and assume that $x_n|\log x_n|\ll K^{-1}$. 
Using Prop. \ref{logdivergencebarrier}, 
\[
\begin{split}
& u_n(x_1,\ldots x_m, x_{m+1} ,\ldots x_i+ Kx_n |\log x_n|^{1/2},x_{i+1}, \ldots x_n) - u_m(x_1,\ldots x_m)
\\
\leq & -C(\delta)^{-1} Kx_n |\log x_n| +C'(\delta) Kx_n |\log x_n|^{1/2},
\end{split}
\]
Recall from (\ref{boundedmodulus2}) that
\[
|u_n(x_1,\ldots x_n)-u_m(x_1,\ldots x_m)|\leq C\sum_{m+1}^n x_k|\log x_k|\leq  Cx_n |\log x_n|.
\]
Since $K\gg C(\delta)$, we can absorb the $Cx_n|\log x_n|$ term, so
\[
\begin{split}
& \partial_i u_n(x_1,\ldots x_n) Kx_n|\log x_n|^{1/2}
\\
  \leq & u_n(x_1,\ldots x_m, x_{m+1} ,\ldots x_i+ Kx_n |\log x_n|^{1/2},x_{i+1}, \ldots x_n)- u_n(x_1,\ldots x_n) 
\\
 \leq & - \frac{1}{2} KC(\delta)^{-1} x_n |\log x_n| +C'(\delta) Kx_n |\log x_n|^{1/2}.
\end{split}
\]
This implies the gradient divergence bound with modified constants.
\end{proof}

\begin{cor}\label{nosubgradient}
Let $n\geq 2$.
The convex function $u_n$ admits no subgradient at any point $x\in \partial \Delta_n$ except for the vertices.
\end{cor}



\subsection{Gradient bounds}

\begin{lem}\label{gradientatworstlog}
Near the boundary \emph{the gradient can diverge at most logarithmically}: in the region
\[
(x_1, \ldots x_m)\in \Delta_m^\delta,\quad 0< x_{m+1},\ldots ,x_n\ll \delta, 
\]
we have 
\[
|\partial_i u_n|\leq C(\delta), \quad i=1,2,\ldots m,\quad  C\log x_i\leq \partial_i u \leq C(\delta),\quad i=m+1,\ldots n.
\]
\end{lem}

\begin{proof}
By the boundedness and convexity of $u_n$, we easily see that
\[
-C(\delta)\leq \partial_iu \leq C(\delta),\quad i=1,\ldots m.
\]
Now we apply the boundary modulus of continuity (\cf Theorem \ref{Existence}), together with convexity, to deduce
\[
C(\delta)\geq \partial_n u\geq \frac{u_n (x_1, x_2,\ldots x_n )- u_n(x_1, x_2,\ldots x_{n-1}, 0) }{x_n} \geq C \log x_n. 
\]
Similarly for $i=m+1,\ldots n$.
\end{proof}

This gradient bound can be improved in a way which sheds some light on the inductive structure in $n$: when we are very close to a lower dimensional face, then the directional derivatives tangential to the face would be close to the gradient in the lower dimensional problem.

\begin{lem}\label{tangentialgradient1}
In the region with
\[
(x_1,\ldots x_m)\in \Delta_m^\delta,\quad 0<x_{m+1},\ldots x_n\ll \delta,
\]
we have for $i\leq m$,
\[
|\partial_i u_n(x_1,\ldots x_n)- \partial_i u_m(x_1,\ldots x_m)| \leq C(\delta) (\max_{k\geq m+1} x_k|\log x_k| )^{1/2}.
\]
Here $C(\delta)$ denotes some constant which depends only on $n, g, \delta$.
\end{lem}

\begin{proof}
Without loss $x_n=\max_{k\geq m+1} x_k\ll 1$.
Denote $y_i= \partial_i u_n(x)$ at the given point $x\in \text{Int}(\Delta_n)$. By convexity,
\[
u_n(x')- u_n(x) \geq \sum_1^n y_i (x_i'-x_i). 
\]
We consider $x'\in \Delta_m$, so that $x_{m+1}'=\ldots= x_n'=0$. Then by Lemma \ref{gradientatworstlog},
\[
u_n(x')- u_n(x) \geq \sum_1^m y_i (x_i'-x_i)- \sum_{m+1}^n y_ix_i\geq \sum_1^m y_i (x_i'-x_i) -C(\delta)\sum_{m+1}^n x_i .
\]
By Theorem \ref{Existence},
\[
|u_n(x)- u_m(x_1,\ldots x_m)| \leq -C \sum_{m+1}^n x_i\log \frac{x_i}{2\pi}\leq Cx_n|\log x_n|.
\]
By the boundary prescription $u_n(x')=u_m(x')$, so combining the above,
\begin{equation*}
u_m(x')- u_m(x_1,\ldots x_m) \geq \sum_1^m y_i (x_i'-x_i)+ C(\delta)x_n \log x_n.
\end{equation*}
By the interior smoothness of $u_m$ (\cf Theorem \ref{Existence}), we can find some ball in $\Delta_m$ of radius comparable to $\delta$, such that the Hessian of $u_m$ has an upper bound $C(\delta)$.
This gives
\[
u_m(x')- u_m(x_1,\ldots x_m) \leq \sum_1^m \partial_i u_m(x_1,\ldots x_m) (x_i'-x_i) + C(\delta) \sum_1^m |x_i-x_i'|^2,
\]
for all $x'$ within the ball. Contrasting the two bounds gives 
\[
\sum_1^m (y_i-\partial_i u_m(x_1,\ldots x_m) )(x_i'-x_i)\leq C(\delta) \sum_1^m |x_i-x_i'|^2+ C(\delta)x_n|\log x_n|,
\]
which shows the result.
\end{proof}

\begin{cor}\label{tangentialgradientconvergence}
Let $x'$ be an interior point of $\Delta_m\subset \partial \Delta_n$, then for any sequence of $x\in \text{Int}(\Delta_n)$ converging to $x'$, the tangential gradient also converges:
\[
\partial_i u_n(x)\to \partial_i u_m(x'),\quad i=1,\ldots m.
\]
\end{cor}

We now turn the attention to the vertices.

\begin{lem}\label{sandwich}
The set of subgradients at the origin
\[
Du_n(0)=\{ y\in \R^n: u(x)\geq u(0)+ \langle x,y\rangle,\forall x\in \Delta_n  \}
\]
is sandwiched between translated copies of the negative quadrant:
\[
-C(n)(1,\ldots ,1) - (\R_{\geq 0}^n) \subset Du_n(0)\subset - (\R_{\geq 0}^n).
\]
Here $C(n)$ is the large constant from Prop. \ref{vertexcontinuity}. In particular $Du_n(0)$ is a closed convex set with nonempty interior.

Similarly, the set of subgradients $Du_n((0,\ldots, \text{k-th entry } \pi, \ldots 0)$   is sandwiched  between two polyhedral cones in $\R^n$:
\[
\{y_k \geq \max_i(0, y_i) \} +C(n) (0, \ldots 1,\ldots 0)\subset Du_n((\ldots, \pi, \ldots)) \subset \{ y_k \geq \max_i(0, y_i) \}.
\]
\end{lem}

\begin{proof}
Since $u_n=0$ at all the vertices, by convexity $y_i\leq 0$ for all the $y\in Du_n(0)$, namely $Du_n(0)\subset - (\R_{\geq 0}^n)$. For the other inclusion, notice that  Prop. \ref{vertexcontinuity} implies the non-emptiness of 
$Du_n(0)$, and that $x_i\geq 0$ implies that translation in the $-\partial_{y_i} $ direction remains inside $Du_n(0)$.
\end{proof}

\begin{lem}\label{subgradientdisjoint}
Let $n\geq 2$. Then the set of subgradients at the different vertices do not intersect.
\end{lem}

\begin{proof}
If $y$ is a subgradient at two different vertices, then it is a subgradient at any point on the line segment joining the two vertices. This does not exist by Cor. \ref{nosubgradient}.
\end{proof}

\begin{prop}\label{gradientimage}
Let $n\geq 2$.
The gradient image of $\text{Int}(\Delta_n)$  is 
\[
Du_n( \text{Int}(\Delta_n))= \R^n\setminus \bigcup Du_n(\text{vertices}).
\]
The set of limiting points for $du_n(x)$ as $x\in \text{Int}(\Delta_n)$ tends to a vertex, is $\partial Du_n(\text{vertex})$.

\end{prop}

\begin{proof}
First we observe that any $y\in \R^n$ is a subgradient at some $x\in \Delta_n$. This is because the graph of the affine linear function $\langle x,y\rangle +a$ has to touch the graph of $u_n$ as $a$ increases from negative infinity. By Cor. \ref{nosubgradient}, this $x$ is either an interior point in $\Delta_n$, or is one of the vertices. This shows
\[
\R^n= Du_n( \text{Int}(\Delta_n)) \cup \bigcup Du_n(\text{vertices}).
\]
Notice that $Du_n( \text{Int}(\Delta_n))$ does not intersect $Du_n(\text{vertices})$ by the strict convexity of $u_n$ in the interior, so
\[
Du_n( \text{Int}(\Delta_n))= \R^n\setminus \bigcup Du_n(\text{vertices}).
\]
Now the set of subgradients at the vertices are disjoint closed sets (\cf Lemma \ref{subgradientdisjoint}). This implies the claim on the limit of gradients.
\end{proof}

\begin{cor}\label{tropicalhypersurface}
Let $n\geq 2$.
The closure of the gradient image
$
\overline{ Du_n(\text{Int}(\Delta_n))} 
$
stays within bounded distance to the tropical hypersurface
\[
\Gamma_{std}=  \text{non-smooth locus of } \max(0, y_1,\ldots y_n) \subset \R^n_y.
\]
\end{cor}

\begin{proof}
At a given point $y\in \overline{ Du_n(\text{Int}(\Delta_n))}$, without loss we assume $\max(0,\ldots y_n)$ is achieved by zero, namely $y_i\leq 0$. By Lemma \ref{sandwich}, the interior of $Du_n(0)$ contains the open cone
$
-C(n) (1,\ldots 1) - \R_{>0}^n.
$
Prop. \ref{gradientimage} forces $y$ to lie inside its complement, so $\max_i y_i\geq -C(n)$. 
Thus $\text{dist}(y, \Gamma_{std})\leq C(n)$.
\end{proof}

\subsection{Partial Legendre transform}

The gradient divergence near the boundary is an indication that the $x_i$ coordinates are not best suited for studying the asymptotic geometry of the special Lagrangian graph. Instead, we shall use a mixture of the position and momentum coordinates, which analytically corresponds to the \emph{partial Legendre transform}. In the Lagrangian version, this idea already appeared in Matessi \cite{Matessi}.

Given $1\leq m\leq n-1$ and $(x_1,\ldots x_m)\in \text{Int}(\Delta_m)$, we can consider the $(n-m)$-dimensional slice of $\Delta_n$ prescribed by $(x_1,\ldots x_m)$. Partial Legendre transform is simply doing Legendre transform on every such slice, by replacing $x_{m+1},\ldots x_n$ with the $y_{m+1},\ldots y_n$ coordinates, where $y_i= \partial_{x_i} u_n$. 

\begin{lem}
On each constant $(x_1,\ldots x_m)$ slice of $\text{Int}(\Delta_n)$, the 
partial derivative map
$
(x_{m+1},\ldots x_n)\mapsto (y_{m+1},\ldots y_n)
$
is a diffeomorphism onto $\R^{n-m}$.
\end{lem}

\begin{proof}
The partial derivative map is a diffeomorphism onto its image, because $u_n$ restricts to a smooth and strictly convex function on the slice, which is an open convex set.
To see this map is surjective to $\R^{n-m}$, the argument imitates Prop. \ref{gradientimage}. Notice that the slice stays away from the vertices of $\Delta_n$, since $(x_1,\ldots x_m)\in \text{Int}(\Delta_m)$.
\end{proof}

We can then define the partial Legendre transform function
\begin{equation}
u_{n,m}(x_1,\ldots x_m, y_{m+1},\ldots y_n)= \min_{x_{m+1,\ldots x_n}}   (u_n(x)-\sum_{m+1}^n x_ky_k),
\end{equation}
which satisfies
\[
\begin{cases}
  \partial_{x_k }u_{n,m}= y_k,\quad & k=1,\ldots m,
  \\
   \partial_{y_k }u_{n,m}= -x_k,\quad & k=m+1,\ldots n.
\end{cases}
\]
The domain of definition is $\text{Int}(\Delta_m)\times \R^{n-m}$. The function $u_{n,m}$ is convex in $x_1,\ldots x_m$, and concave in $y_{m+1},\ldots y_n$.

\begin{lem}
The special Lagrangian graph equation  in the $x_1,\ldots x_m,y_{m+1},\ldots y_n$ coordinate system reads
$
\arctan D^2 u_{n,m}= \frac{m-1}{2}\pi.
$
\end{lem}

\begin{lem}\label{expdecay1}
We have
\[
\begin{cases}
   x_i \leq e^{y_i/C},\quad i=m+1,\ldots n,
\\
|u_{n,m}(x_1,\ldots x_m, y_{m+1},\ldots)- u_m(x_1,\ldots x_m) |\leq C'\max_{k\geq m+1} e^{y_k/C}.
\end{cases}
\]
The constants only depend on $n, g$.
\end{lem}

\begin{proof}
The global estimate $x_i\leq e^{y_i/C}$ follows from the at worst log divergence of gradient, \cf Prop. \ref{gradientatworstlog}, which comes from the boundary modulus of continuity estimate (\ref{boundarymodulus}), and the constant depends only on $n, g$. By the definition of the Legendre transform
\[
u_{n,m} =u_n(x)-\sum_{m+1}^n x_ky_k ,
\]
so using the boundary modulus of continuity (\ref{boundedmodulus2}),
\[
|u_{n,m}-u_m| \leq |u_n(x)-u_m|+\sum x_k |y_k| \leq C\sum x_k|\log x_k| + \sum x_k |y_k|\leq C'\max_{m+1}^n |y_k| e^{y_k/C}.
\]
The $|y_k|$ factor can be absorbed by the exponential after modifying constants.
\end{proof}

\begin{prop}\label{boundarychartexpdecay}
Given small $\delta>0$, and $k\geq 2, \alpha\in (0,1)$, there exists large positive constants $C(\delta), C'(\delta)$, such that 
on the region
\begin{equation}\label{boundarychartm}
(x_1,\ldots x_m)\in \Delta_m^\delta,\quad y_{m+1},\ldots y_n\leq -C(\delta),
\end{equation}
we have the local higher order derivative estimates
\[
\norm{ u_{m,n}- u_m  }_{C^{k,\alpha}} \leq C(\delta, k,\alpha) \max_{i\geq m+1} e^{y_i/C},
\]
and
\[
\norm{ \partial_{y_i} u_{m,n}   }_{C^{k,\alpha}} \leq C(\delta, k,\alpha)  e^{y_i/C},\quad i=m+1,\ldots n.
\]
\end{prop}

\begin{proof}
The special Lagrangian graph equation satisfied by $u_{n,m}$ is elliptic, and independent of first order derivatives.
The function $u_m$ (independent of the $y_{m+1},\ldots y_n$ variables) is a special solution to the same equation with smoothness bounds on $\Delta_m^\delta\times \R^{n-m}$. Morever, for $y_{m+1},\ldots y_n$ sufficiently negative depending on $\delta$, the $C^0$-difference $|u_{n,m}-u_m|$ can be made arbitrarily small by the exponential decay estimate in Lemma \ref{expdecay1}. Thus on a slightly shrunken domain, we can verify the conditions for Savin's small perturbation theorem \cite{Savin}, to bound $\norm{u_{n,m}-u_m}_{C^{2,\alpha}}$ by a small quantity, on local coordinate balls of radius $O(\delta)$.

Once we know $u_{n,m}$ is $C^{2,\alpha}$-close to $u_m$, the problem linearises, and we can bootstrap the exponentially small $C^0$ bound on $u_{n,m}-u_m$ to $C^{k,\alpha}$ norm.  Similarly, by differentiating the equation we obtain a uniformly elliptic equation on $\partial_{y_i} u_{m,n}=-x_i$. The $C^0$ bound on this quantity is $O(e^{y_i/C})$, which we can bootstrap to higher orders.
\end{proof}

\begin{rmk}
The reason to separately state the higher derivative bound on $\partial_{y_i} u_{m,n}=-x_i$, is that when the $y_i$ are large but not all of comparable order, then some $x_i$ would be much smaller than others.
\end{rmk}

\begin{rmk}\label{localchartrmk}
We will regard (\ref{boundarychartm}) as giving a local chart of the special Lagrangian graph. Lemma \ref{expdecay1} implies that $x$ lies close to the shrunken boundary face $\Delta_m^\delta$. Conversely, for $(x_1,\ldots x_m)\in \Delta_m^\delta$, and $x_{m+1},\ldots x_n$ sufficiently small depending on $\delta$, the gradient divergence estimate in Cor. \ref{gradientdivergence} implies that $y_i=\partial_{x_i} u_n$ is very negative, hence lies inside the chart (\ref{boundarychartm}). These charts cover a $C(\delta)^{-1}$ neighbourhood of the boundary of $\partial \Delta_n$, minus the $O(\delta)$-neighbourhood of the vertices. Of course, the interior region $\{ \text{dist}(x, \partial \Delta_n)\gtrsim C(\delta)^{-1} \}$ is simply described by $u_n$ itself, with its interior smoothness estimates.
\end{rmk}

\section{Geometry of the solution}


\subsection{Special Lagrangian current}

We extend $u_n$ to a function over $\Delta_n\cup -\Delta_n$ as an odd function,
$
u_n(-x)=-u_n(x).
$
Over $\text{Int}(-\Delta_n)$, this odd extension $u_n$ satisfies 
\[
\arctan D^2 u_n= - \frac{(n-1)\pi}{2}.
\]
This differs from $\hat{\theta}=\frac{n-1}{2}\pi$ by an integer multiple of $\pi$. Thus under the orientation convention that $\text{Re}(e^{-\sqrt{-1}\hat{\theta}}\Omega)>0$, the graph of the gradient $du_n$ on $\text{Int}(\Delta_n\cup (-\Delta_n))$
defines a smooth special Lagrangian  $\mathbb{L}^o\subset T^*T^n$, with phase angle $\hat{\theta}$.

The main drawback is that $\mathbb{L}^o$ is not yet a closed subset of $T^*T^n$. To remedy this, we view $\mathbb{L}^o$ as defining an integration current $\mathbb{L}$. The support of $\mathbb{L}$ contains also the points on the closure of $\mathbb{L}^o$. As a current $\mathbb{L}$ is sent to $-\mathbb{L}$ under the involution 
\begin{equation}\label{involution}
x_i\to -x_i, \quad y_i\to y_i.
\end{equation}
We will use the notation $\mathbb{L}_n$ when we wish to emphasize on the $n$-dependence.

\begin{prop}\label{integralcurrentnoboundary}
The current $\mathbb{L}$ is a special Lagrangian integral current of phase $\hat{\theta}$, and $\partial \mathbb{L}=0$. The volume growth has the upper bound $\text{Mass}(\mathbb{L}\cap B(R))\leq CR^{n-1} $.
\end{prop}

\begin{proof}
We first show that $\mathbb{L}$ has locally finite mass in $T^*T^n$. For simplicity we first treat the special case $g_{ij}=\delta_{ij}$, and consider a large ball $B(R)\subset T^*T^n$. By the special Lagrangian property,
\[
\text{Mass}(\mathbb{L}\cap B(R))=\int_{ \mathbb{L}^o\cap B(R) } \text{Re} ( e^{-\sqrt{-1}\hat{\theta}} \Omega) = \int_{ \mathbb{L}^o\cap B(R) }  \text{Re} ( -\sqrt{-1}\bigwedge_1^n (dy_i-\sqrt{ -1} dx_i) ),
\]
which is a weighted sum of terms like
\[
\int_{ \mathbb{L}^o\cap B(R) }  dy_{i_1}\wedge \ldots dy_{i_k}\wedge dx_{i_{k+1}}\ldots \wedge dx_{i_n},\quad k<n.
\]
Notice here that at least one $dx_i$ factor appears in the integrand.
Now $y_i= \partial_{x_i} u_n$, and $u_n$ is a smooth and strictly convex function on $\text{Int}(\Delta_n)$.  Thus the projection from the portion of $\mathbb{L}^o$ over $\text{Int}(\Delta_n)$, to the $n$-dimensional plane with coordinates $y_{i_1},\ldots x_{i_n}$, is diffeomorphic to its image.  This image is contained in 
$\Delta_{n-k}\times \{ |y_i|\leq R, \forall i\leq k  \}$, so the volume integral is bounded by $CR^k\leq CR^{n-1}$. More generally, we can reduce to the $g_{ij}=\delta_{ij} $ case by a linear change of coordinates (\cf section \ref{specialLaggrapheqnsection}).


Next we analyse the boundary of the current $\mathbb{L}$. From Cor. \ref{gradientdivergence} we know the  divergence of the gradient $du_n$ near all the boundary strata except for the vertices, so the only possible contribution to $\partial \mathbb{L}\cap B(R)$ comes from the neighbourhood of the vertices. We focus on the vertex at the origin. By the same kind of volume computation as above, for small $r\ll 1$ and fixed $R$, we have
\[
\text{Mass}(\mathbb{L}^o\cap \{ |x_i|<r, |y_i|\leq R, \forall i) \leq CR^{n-1}r.
\]
Morever, the integral of forms involving at least two $dx_i$ factors are suppressed by $O(R^{n-2} r^2)$.

We can pick a standard cutoff function $\chi$ which is supported on $\{ |x_i|<r,\forall i  \}$ (and similarly on the neighbourhood of the other vertices), and satisfies $|d\chi|\leq Cr^{-1}$. Then for any fixed $(n-1)$-form $\alpha$ compactly supported inside $B(R)$, we have
\[
\begin{split}
& \int_{\mathbb{L}^o} d(\chi\alpha) \leq \norm{d\alpha}_{C^0} CR^{n-1} r + \sum \norm{\alpha(\partial_{y_{i_1}}, \ldots \partial_{y_{i_{n-1}}})}_{C^0}  CR^{n-1} r\norm{d\chi}_{C^0}
\\
&+ \norm{\alpha}_{C^0} CR^{n-2}r^2\norm{d\chi}_{C^0} 
\\
& \leq C\norm{\alpha}_{C^1} R^{n-1} r+ \sum \norm{
\alpha(\partial_{y_{i_1}}, \ldots \partial_{y_{i_{n-1}}})}_{C^0}  CR^{n-1} .
\end{split}
\]
Observe that
\[
\int_{\mathbb{L}^o} d(  (1-\chi)\alpha) =0
\]
since $(1-\chi)\alpha$ is supported away from the vertex neighbourhoods. Thus
\[
\int_{\mathbb{L}^o} d\alpha \leq C\norm{\alpha}_{C^1} R^{n-1} r+ 
\sum \norm{
\alpha(\partial_{y_{i_1}}, \ldots \partial_{y_{i_{n-1}}})}_{C^0}
  CR^{n-1} .
\]
Taking the $r\to 0$ limit, this shows 
\[
\int_{\mathbb{L}^o} d\alpha \leq \sum \norm{
\alpha(\partial_{y_{i_1}}, \ldots \partial_{y_{i_{n-1}}})}_{C^0}  CR^{n-1} .
\]
Thus $\partial \mathbb{L}$ has finite mass within the ball $B(R)$. Clearly $\partial \mathbb{L}$ is supported in the cotangent fibres over the  vertices of $\Delta_n$, and by the Riesz representation theorem it can be written as a $(n-1)$-form with signed measure coefficients on these cotangent fibres. We emphasize that only $dy_i$ type factors appear, and the $dx_i$ factors are killed.

Now $\partial \mathbb{L}$ is sent to $-\partial \mathbb{L}$ under the involution (\ref{involution}). Since $dy_i$ are invariant under the involution, we conclude that the measure coefficients vanish. This means that $\partial \mathbb{L}=0$ within any fixed ball. 
\end{proof}

\begin{rmk}
Geometrically, the boundary contributions from the vertices of $\Delta_n$ and $-\Delta_n$ are both nonzero, but cancel out. It is an instructive excercise to see how this works in $n=1,2$.
\end{rmk}

\begin{prop}\label{supportL}
The support of $\mathbb{L}$ is the closure of $\mathbb{L}^o$, whose projection to $T^n$ is contained in the union of $\text{Int}(\Delta_n\cup -\Delta_n)$ and the vertices of $\Delta_n$. The set of limiting points of $\mathbb{L}$ inside the cotangent fibre over the origin, is a copy of $\partial Du_n(0) \subset \R^n_y$, and similarly with the other vertices.
\end{prop}

\begin{proof}
By the divergence of the gradient in Prop. \ref{gradientdivergence}, there is no limiting point over the boundary of $\Delta_n\cup -\Delta_n$ except at the vertices. The claim on the limiting points is a reformulation of Prop. \ref{gradientimage}.
\end{proof}

\subsection{Smoothness of the special Lagrangian}

Our goal is to show

\begin{prop}\label{smoothness}
The special Lagrangian current $\mathbb{L}_n$ is supported on a smooth embedded submanifold (also denoted $\mathbb{L}_n$ henceforth). 
\end{prop}

By the interior smoothness of $u_n$, the essential task is to prove smoothness at the points over the vertices.  By the Allard regularity theorem, it suffices to show

\begin{prop}\label{tangentcone}
Let $(0,\mathfrak{y})\in \text{supp}(\mathbb{L})$, then any tangent cone $\mathcal{C}$ at $(0,\mathfrak{y})$ is a multiplicity one $n$-plane. In fact, it projects to a one-dimensional line in the $x$-plane. 
\end{prop}

The following Lemma is inspired by Joyce's work \cite[Lemma 5.7]{JoyceU1}

\begin{lem}(Multiplicity one projection)\label{multoneprojection}
Let $(a_{ij})$ be an $n\times n$ matrix with $\det>0$, and the last row entries $a_{nj}\geq 0$. Let $(a^{ij})$ be the inverse matrix, and write the rotated coordinates
\[
\tilde{x}_i=a_{ij} x_j,\quad \tilde{y}_i= a^{ji} y_j.
\]
Let $\chi(\tilde{y}_1,\ldots \tilde{y}_{n-1}, \tilde{x}_n)$ be any nonnegative function on $\R^n$, compactly supported inside $\{ \tilde{x}_n>0\}$. 
Then for any  tangent cone $\mathcal{C}$ at $(0,\mathfrak{y})\in \text{supp}(\mathbb{L})$, the projection map from $\mathcal{C}$ to the $(\tilde{y}_1,\ldots \tilde{y}_{n-1}, \tilde{x}_n)$ coordinate $n$-plane  has non-negative Jacobian, and  
\begin{equation}\label{tangentconedegone}
\int_{\mathcal{C}} \chi d\tilde{y}_1\wedge\ldots d\tilde{y}_{n-1}\wedge d\tilde{x}_n \leq \int_{\R^n } \chi d\tilde{y}_1\ldots d\tilde{y}_{n-1}d\tilde{x}_n.
\end{equation}

\end{lem}

\begin{proof}
First, we notice that in a fixed small ball $B_{(0,\mathfrak{y})}(r)$, the projection map from
$\mathbb{L}\cap B_{(0,\mathfrak{y})}(r)$ to the $\R^n$ plane with the $\tilde{y}_1,\ldots \tilde{y}_{n-1}, \tilde{x}_n$ variables, is an injective map over $\{ \tilde{x}_n>0\}$. 
To see this, observe that $a_{nj}\geq 0$ and $x\in \Delta_n\cup -\Delta_n$ forces the preimage to lie on the part of $\mathbb{L}$ defined by the graph of $du_n$ over $\text{Int}({\Delta}_n)$. In these rotated coordinates, the Lagrangian $\mathbb{L}$ satisfies $\tilde{y}_i= \frac{\partial u_n}{\partial \tilde{x}_i}$ for $i=1,\ldots n$, so by the strict convexity of $u_n$, the preimage is unique. This shows
\begin{equation*}
\int_{\mathbb{L}\cap  B_{(0,\mathfrak{y})}(r) } \chi d\tilde{y}_1\wedge\ldots d\tilde{y}_{n-1}\wedge d\tilde{x}_n \leq \int_{\R^n } \chi d\tilde{y}_1\ldots d\tilde{y}_{n-1}d\tilde{x}_n.
\end{equation*}
for any non-negative function  $\chi(\tilde{y}_1,\ldots \tilde{y}_{n-1}, \tilde{x}_n)$ on $\R^n$, compactly supported inside $\{ \tilde{x}_n>0\}$. 
The same inequality applies to any rescaling of $\mathbb{L}$ around $(0,\mathfrak{y})$, where the ball $B_{(0,\mathfrak{y})}(r)$ needs to be dilated along with $\mathbb{L}$. As the dilation scale goes to infinity, we can pass the inequality to the limit to deduce the tangent cone statement (\ref{tangentconedegone}).

Morever by the special Lagrangian graph equation and the orientation convention for $\mathbb{L}$,  the $n$-form $d\tilde{y}_1\wedge\ldots d\tilde{y}_{n-1}\wedge d\tilde{x}_n$ is non-negative on $\mathbb{L}\cap  B_{(0,\mathfrak{y})}(r)$. This property also passes to the tangent cone, and implies the non-negativity of the  Jacobian of the projection $\mathcal{C}\to \R^n$. 
\end{proof}

We now prove Prop. \ref{tangentcone}.

\begin{proof}
Let $\mathcal{C}$ denote a tangent cone at $(0,\mathfrak{y})$, then $\mathcal{C}$ is a special Lagrangian current of phase $\hat{\theta} $, hence locally area minimizing. By Almgren's codimension two regularity, $\mathcal{C}$ can be decomposed into irreducible components, which are individually locally area minimizing  currents without boundary, such that the smooth locus of each component is path connected. Clearly each component $\mathcal{C}'$ is also a minimal cone, with the cone apex at the origin inside $T_{(0,\mathfrak{y})}(T^*T^n)$.

Since the support of $\mathbb{L}$ is contained in the coamoeba, we know the $x$-projection of the tangent cone $\mathcal{C}$ is contained in the union of the two quadrants $\R_{\geq 0}^n \cup - \R_{\geq 0}^n $. Clearly the support of $\mathcal{C}$ is invariant under the involution (\ref{involution}). Notice that the cotangent fibre $\R^n_y$ is a special Lagrangians with phase $\frac{n}{2}\pi$, which is different from $\hat{\theta}= \frac{n-1}{2}\pi$. Thus $\mathcal{C}$ has no $n$-dimensional component contained inside $\R^n_y$, so each component $\mathcal{C}'$ contains some point with $x$-coordinate in $ \R_{\geq 0}^n\setminus \{ 0\}$. Without loss the $n$-th coordinate is positive.

Consider the link of $\mathcal{C}'\cap \{ x_i\geq 0, \forall i\}$. For each $i=1,\ldots n-1$, as the parameter $a_i$ increases from zero to infinity, the family of hyperplanes $\{ a_i x_n=x_i\}$ must first touch this link at some parameter value $a_i$. 
 In a unit ball around the point of touching, $\mathcal{C}'$ locally lies inside the half space $\{ x_i-a_ix_n\geq 0 \}$, and the touching point lies on the boundary of the half space. By the strong maximum principle (\cf Prop. \ref{strongmax}), $\mathcal{C}'$ has nontrivial $n$-dimensional measure inside the hyperplane $\{ x_i= a_i x_n\}$, so by the path connectedness of its smooth locus, the whole component $\mathcal{C}'$ lies inside the intersection of the hyperplanes $\{ x_i= a_i x_n\}$, for $i=1,\ldots n-1$.
Thus the $x$-projection of $\mathcal{C}'$ is the line along the direction $\sum_1^{n-1} a_i \partial_{x_i}+ \partial_{x_n}$. Since $\mathcal{C}'$ is a Lagrangian cone, this implies $\sum_1^{n-1} a_iy_i+y_n=\text{const}$. In summary, any tangent cone component $\mathcal{C}'$ is supported on some $n$-plane of the form 
\begin{equation}\label{tangentspace}
(x_1,\ldots x_n)\in \R\text{-span}(a_1,\ldots a_n) ,\quad \sum_1^n a_iy_i=0,\quad (a_i)\in \R_{\geq 0}^n\setminus \{ 0\}.
\end{equation}
By the constancy theorem, $\mathcal{C}'$ is some integer multiple of the $n$-plane, and $\mathcal{C}$ is a sum of such components.

Now by Lemma \ref{multoneprojection}, for any choice of the parameter matrix 
$(a_{ij})$, the projection map of $\mathcal{C}$ into the $(\tilde{y}_1,\ldots \tilde{y}_{n-1}, \tilde{x}_n)$ coordinate $n$-plane has non-negative Jacobian, and the projection degree (counting multiplicity) is at most one. 
This forces $\mathcal{C}$ to have only one component, with multiplicity one.
\end{proof}

We can be a little more precise on the local structure.

\begin{cor}\label{graphicalnearvertexlem}
The tangent plane at any given point $(0,\mathfrak{y})\in \text{supp}(\mathbb{L})$ is of the form 
\[
(x_1,\ldots x_n)\in \R\text{-span}(a_1,\ldots a_n) ,\quad \sum_1^n a_i(y_i-\mathfrak{y}_i)=0,\quad (a_i)\in \R_{\geq 0}^n\setminus \{ 0\}.
\]
Suppose $a_n=1$. Take the $n\times n$ matrix 
\[
(a_{ij})= 
\begin{pmatrix}
1 & 0 & \cdots & -a_1 \\
0 & 1 & \cdots & -a_2 \\
\vdots  & \vdots  & \ddots & \vdots  \\
0 & 0 & \cdots & 1 
\end{pmatrix}
\]
with inverse matrix $(a^{ij})$,  and write the rotated coordinates 
\[
\tilde{x}_i=a_{ij} x_j,\quad \tilde{y}_i= a^{ji} y_j.
\]
Then there are real analytic functions $f_i$ 
defined on a small coordinate ball around the origin in $\R^n$, with $f_i(0)=0$, such that around $(0,\mathfrak{y})$ the special Lagrangian is locally given by
\begin{equation}\label{graphicalnearvertex}
\begin{cases}
   \tilde{x}_i= \tilde{x}_n f_i( \tilde{x}_n, \tilde{y}_1- \tilde{\mathfrak{y}}_1,\ldots \tilde{y}_{n-1}- \tilde{\mathfrak{y}}_{n-1}),\quad i=1,\ldots n-1.
   \\
   \tilde{y}_n= \tilde{\mathfrak{y}}_n +f_n(\tilde{x}_n, \tilde{y}_1- \tilde{\mathfrak{y}}_1,\ldots \tilde{y}_{n-1}- \tilde{\mathfrak{y}}_{n-1}).
\end{cases}
\end{equation}
In particular, the boundary of the convex set $\partial Du_n(0)$ is locally given by 
\[
 \tilde{y}_n=\tilde{\mathfrak{y}}_n +f_n(\tilde{x}_n, \tilde{y}_1- \tilde{\mathfrak{y}}_1,\ldots \tilde{y}_{n-1}- \tilde{\mathfrak{y}}_{n-1})
\]
and is therefore a \emph{real analytic hypersurface} in the cotangent fibre $\R^n_y$.

\end{cor}

\begin{proof}
The tangent space of $\mathbb{L}$ can be read off from the tangent cone (\ref{tangentspace}), after translating the origin back to $(0, \mathfrak{y})$. In the rotated coordinates, the tangent space is simply 
\[
\tilde{x}_1=\ldots =\tilde{x}_{n-1}=0,\quad \tilde{y}_n=\tilde{\mathfrak{y}}_n.
\]
Since $\mathbb{L}$ is a smooth special Lagrangian in the Euclidean space, it must be a real analytic submanifold, so locally around $(0,\mathfrak{y})$ it can be written as the graph of an analytic vector-valued function over the tangent space. Morever, when $\tilde{x}_n=x_n=0$, then the point on the special Lagrangian lies over the vertex of $\Delta_n$ at the origin, so $\tilde{x}_i=0$ for all $i=1,\ldots n-1$. Thus the Taylor expansion of the real analytic function $\tilde{x}_i$ in the $\tilde{x}_n$ variable has no constant term, so we obtain the graphical representation (\ref{graphicalnearvertex}). From the tangent cone information, we can read off $f_i(0)=0$.

The claim on $\partial Du_n(0)$ follows by recalling from Prop. \ref{supportL} that the part of $\mathbb{L}$ inside the cotangent fibre $\R^n_y$ over the origin, is a copy of $\partial Du_n(0)$.
\end{proof}

\subsection{Asymptotic geometry}\label{asymptoticproof}

We now prove Theorem \ref{asymptotic} by induction on $n$. We already know the $n=2$ case from the explicit solution, so we will assume that $n\geq 3$ and the statement holds already for all lower dimensions. In particular, the $C^{k,\alpha}$ topology makes sense on $\mathbb{L}_{n-1}.$

\begin{lem}\label{awayfromvertex}
There is some $C(\delta), C(\delta)'$ large enough depending on $n,g,\delta$, such that the portion of the special Lagrangian
\[
\mathbb{L}_n\cap \{ y_n\leq -C(\delta)' \} \cap \{ \text{dist}(x, \text{vertices})\gtrsim \delta  \}
\]
lies on a $C^{k,\alpha}$-small graph over $\mathbb{L}_{n-1}\times \R_{y_n}\subset T^*T^{n-1}\times T^*S^1=T^*T^n$. The local $C^{k,\alpha}$-norm of the graph is bounded by $C(\delta) e^{ y_n/C(n)}$, where the exponential decay rate $C(n)$ is independent of $\delta$.

\end{lem}

\begin{proof}
By Lemma \ref{expdecay1}, the $x_n$ coordinate is exponentially small when $y_n$ is very negative. Recall from Remark \ref{localchartrmk} that for any given small $\delta>0$, we used the partial Legendre transform $u_{n,m}$ to assign smooth charts for $\mathbb{L}_n$, except in the $O(\delta)$-neighbourhood around the vertices.

Now $\mathbb{L}_{n-1}$ is also assigned with very similar coordinate charts by the partial Legendre transform $u_{n-1, m}$ for $u_{n-1}$ involving fewer variables. We regard $u_{n-1,m}$ as a function of $n$ variables $x_1,\ldots , x_m, y_{m+1},\ldots y_n$, with no actual dependence on the $y_n$ coordinate. Geometrically, this corresponds to taking the product of $\mathbb{L}_{n-1}$ with $\R_{y_n}$. By construction \[
\arctan D^2 u_{n,m}= \arctan D^2u_{n-1,m}= \frac{(m-1)}{2}\pi.
\] 
We comment that for $m=n-1$, the function $u_{n-1,m}$ is just $u_{n-1}$, and the corresponding chart is meant to cover the interior region in $\Delta_{n-1}$. Our task is to compare $u_{n,m}$ with $u_{n-1,m}$ on the corresponding charts.
By Prop. \ref{boundarychartexpdecay} both functions have higher derivative estimates depending on $\delta$ in the chart, and the main task is to prove the exponential decay estimate.

Let $v=u_{n,m}- u_{n-1,m}$. We can show 
\[
|v|\leq C' e^{y_n/C}
\]
by the same argument as in Lemma \ref{expdecay1}. We will then derive an elliptic equation on $v$. Applying the fundamental theorem of calculus to the function $\arctan $ on symmetric matrices,
\[
\arctan (M+N)-\arctan M= \int_0^1 \Tr ((M+tN)^2 +I)^{-1}   N) dt,
\]
hence
\[
0= \arctan D^2u_{n,m}- \arctan D^2 u_{n-1, m}= \int_0^1 \Tr ((D^2 u_{n-1,m}+tD^2 v)^2 +I)^{-1}D^2 v ) dt.  
\]
Since $D^2 v$ and $D^2 u_{n-1,m}$ are already bounded, this is a uniformly elliptic equation, so the $C^0$ exponential decay on $v$ can be bootstrapped to all higher derivatives. This implies the geometric statement.
\end{proof}

We will need a quantitative version of Allard regularity theorem \cite{Allard}.

\begin{prop}
    (Allard regularity)
There is a small universal constant $\epsilon_0\ll 1$ depending only on $n, N$ such that the following holds,
Let $X$ be an $n$-dimensional multiplicity one stationary integral varifold inside $B(p,r) \subset \R^N$. Assume that $p$ lies on the support of $X$, and the volume $\mathcal{H}^n(X\cap B(p,r)) \leq \omega_n(1+\epsilon_0)$, then 
$X\cap B(p, r/2)$ is a $C^{1,\alpha}$ graph over the tangent plane through $p$, with $C^{1,\alpha}$ norm bounded by $\frac{1}{100}$.
\end{prop}

It remains to have quantitative control of the geometry in the neighbourhood of the vertices. 
The same argument as in Prop. \ref{integralcurrentnoboundary} gives the following:

\begin{lem}\label{vertexvolume}
Take any point $p\in \mathbb{L}_n$, and let $\delta\lesssim r\leq 1$. Then 
\[
\text{Vol}(\mathbb{L}_n\cap \{ \text{dist}(x, \text{vertex})\lesssim \delta  \} \cap B(p, r) )\leq Cr^{n-1}\delta.
\]

\end{lem}

We now prove Part 1 of Theorem \ref{asymptotic} on regularity.

\begin{proof}
By the $C^{k,\alpha}$ regularity of $\mathbb{L}_{n-1}$ in the inductive hypothesis, we have some radius parameter $r_0<1$ depending on $n,g$, such that the volume ratio for $\mathbb{L}_{n-1}\times \R$ in any ball of radius $\leq r_0$ around any point on $\mathbb{L}_{n-1}\times \R$, is bounded by $1+ \frac{\epsilon_0}{3}$, where $\epsilon_0$ is the Allard regularity constant.

We choose $\delta\ll \epsilon_0r_0$, and consider the points  $p=(x',y')\in \mathbb{L}_n$.
Using Lemma \ref{awayfromvertex}, for $y'\leq -C(\delta)'$ very negative, we have
\[
\text{Vol}(\mathbb{L}_n\cap B(p, r_0)\cap \{ \text{dist}(x, \text{vertex}) \gtrsim \delta) \} ) \leq \omega_n r_0^n(1+ \frac{\epsilon_0}{3}+ C(\delta) e^{y'/C}).
\]
By taking $y'\leq -C(n,g, \epsilon_0, \delta)$ sufficiently negative, we can bound the RHS by $\omega_n r_0^n(1+ \frac{\epsilon_0}{2})$. 
On the other hand, by Lemma \ref{vertexvolume} and the choice of $\delta$,
\[
\text{Vol}(\mathbb{L}_n\cap \{
\text{dist}(x, \text{vertex})
\lesssim \delta  \} \cap B(p, r_0) )\leq Cr_0^{n-1}\delta\ll r_0^n \epsilon_0.
\]
Summing over the two contributions,
\[
\text{Vol}(\mathbb{L}_n \cap B(p, r_0) )\leq \omega_n r_0^n (1+\epsilon_0),
\]
so Allard regularity gives that $\mathbb{L}_n\cap B(p, r_0/2)$ is a $C^{1,\alpha}$ graph over the tangent plane through $p=(x',y')$, with $C^{1,\alpha}$ norm bounded by $\frac{1}{100}$, which can be bootstrapped to $C^{k,\alpha}$-estimates by applying Schauder estimates to the minimal surface system.

We have successively chosen the constants $\epsilon_0, r_0, \delta, C(n, g, \epsilon_0, \delta)$, so all these constants only depend on $n,g$ in the end. The upshot is that we proved the quantitative $C^{k,\alpha}$ regularity on balls of radius comparable to $r_0$, in the region on $\mathbb{L}_n$ where $y_n$ is sufficiently negative depending only on $n,g$. The same argument works when one of $y_1,\ldots y_n, -\sum_1^n y_i$ is very negative, corresponding to the other codimension one faces of $\Delta_n$. These regions cover all but a compact set on $\mathbb{L}_n$. But Prop \ref{smoothness} gives the smoothness of $\mathbb{L}_n$, which provides the $C^{k,\alpha}$ boundedness on any compact region. 
\end{proof}

\begin{cor}\label{Ln-1chart}
There is a large enough $C(n)$ depending on $n,g$,  such that 
the portion of the special Lagrangian
\[
\mathbb{L}_n\cap \{ y_n\leq -C(n) \} 
\]
lies on the graph of a normal vector field $v$ over $\mathbb{L}_{n-1}\times \R_{y_n}\subset T^*T^n$, with $C^{k,\alpha}$ norm bounded by $C(n,k,\alpha)$.
\end{cor}

\begin{proof}
We use the notations in the above proof. By Lemma \ref{awayfromvertex} and the above proof, $\mathbb{L}_n\cap B(p, r_0/2)$ is a $C^{1,\alpha}$-small graph over $\mathbb{L}_{n-1}\times \R$ away from the small subset $\{ \text{dist}(x, \text{vertex})\lesssim \delta\}$, as well as a small $C^{1,\alpha}$-small graph over the tangent plane at $p$. Thus the tangent plane is itself a $C^{1,\alpha}$-small graph over $\mathbb{L}_{n-1}\times \R \cap B(p, r_0/3)$, and so must be $\mathbb{L}_n\cap B(p, r_0/3)$. 
\end{proof}

We now prove Part 2 of Theorem \ref{asymptotic} on the inductive asymptote.

\begin{proof}
We are left to prove the exponential decay in the inductive asymptote statement of Theorem \ref{asymptotic}, namely that for $y_n$ sufficiently negative,  the normal vector field $v$ in Cor. \ref{Ln-1chart} has local $C^{k,\alpha}$ norm bounded by $C' e^{y_n/C(n)}$ for constants depending only on $n,g$.

We define for sufficiently large $R$
\[
f(R)= \sup_{p=(x,y), y_n\leq -R  } \norm{v}_{ C^{k,\alpha}(B(p,r_0)\cap \mathbb{L}_{n-1}\times \R)}. 
\]
Now $v$ inherits an elliptic equation from the minimal surface, and we already know the boundedness of $C^{k,\alpha}$-norms, so for $p\in \mathbb{L}_{n-1}\times \R_{y_n}$ with $y_n\leq -R-r_0$,
\[
 \norm{v}_{ C^{k,\alpha}(B(p,r_0)\cap \mathbb{L}_{n-1}\times \R)}\leq C \dashint_{ B(p,r_0)\cap \mathbb{L}_{n-1}\times \R } |v|.
\]
The $L^1$ average can be split into the $\{ \text{dist}(x, \text{vertex})\gtrsim \delta\}$ region contribution, which is $O( C(\delta)e^{y_n/C(n)})$ by Lemma \ref{awayfromvertex}, and the $\{ \text{dist}(x, \text{vertex})\lesssim \delta\}$ region contribution, which is $O( \delta r_0^{-1} f(R)) $, using the smallness of volume estimate from Lemma \ref{vertexvolume}. In summary, there are constants depending only on $n,g$ such that
\[
f(R+ r_0) \leq C' e^{-R/C(n)} + (C' \delta r_0^{-1}) f(R)  .
\]
Choosing $\delta\ll r_0$, the coefficient of $f(R)$ is strictly smaller than one. Iterating this estimate gives exponential decay on $f(R)$, namely the desired exponential decay estimate on the local $C^{k,\alpha}$-norms. 
\end{proof}

\subsection{Topology of the special Lagrangian}\label{Topologysection}

We recall the notion of \emph{real blow up} from the work of Matessi \cite{Matessi}. We start from 
the coamoeba $C_{std}$ (\cf(\ref{coamoeba})).  At each vertex, the tangent cone of $C_{std}$ is a copy of $\R_{>0}^n \cup -\R_{>0}^n$, and the set of real lines contained in this tangent cone is an open subset of the real projective space $\mathbb{RP}^{n-1}$. The real blow up is then
\[
\tilde{C}_{std}= \text{Int}(\Delta_n\cup -\Delta_n)\cup \bigcup_{vertices} \{ l\in \mathbb{RP}^{n-1}: \text{$l$ is contained in the tangent cone}  \}.
\]
This can be given a natural real analytic structure.
For instance, near the origin in $\Delta_n$, the ratios $\frac{x_i}{x_j}$ extend to local analytic functions on the real blow up.

Our goal is to show

\begin{prop}
  \label{Topology}
Let $n\geq 2$.
The natural projection $\mathbb{L}\to C_{std}$ lifts to a real analytic map $\mathbb{L}\to \tilde{C}_{std}$, which is a homeomorphism.   
\end{prop}

\begin{lem}\label{normalvectorhypersurface}
At any point on the boundary of the convex set $Du_n(0)\subset \R^n_y$, the outward pointing normal vector lies in $\R_{>0}^n\subset \R^n_x= (\R^n_y)^*$. 
\end{lem}

\begin{proof}
By definition $Du_n(0)$ is the set of subgradients at the vertex. Since $\Delta_n$ is contained in $\R_{\geq 0}^n$, translation in the $-\R\partial_{y_i} $ direction remains inside $Du_n(0)$. Let $\nu$ denote the outward pointing normal at any given point $\mathfrak{y}\in \partial Du_n(0)$, then $\langle \nu, -\partial_{y_i} \rangle\leq 0$ for any $i=1,\ldots n$, namely $\nu\in \R_{\geq 0}^n$.

Suppose for contradiction that one of the coordinates of $\nu$ is zero, say $\nu_n=0$. By the convexity of the set $Du_n(0)$, 
\[
\langle y-\mathfrak{y}, \nu\rangle \leq 0,\quad \forall y\in Du_n(0). 
\]
Notice that for any $a\geq 0$, we have
\[
y=\mathfrak{y}- a \partial_{y_n} \in Du_n(0),\quad \langle y-\mathfrak{y}, \nu\rangle = 0.
\]
so the ray $\mathfrak{y}- \R_{\geq 0}\partial_{y_n}$ is contained inside the boundary $\partial Du_n(0)$. But $\partial Du_n(0)$ is a real analytic hypersurface, so by analytic continuation, the entire real line $\mathfrak{y}+ \R\partial_{y_n}$ is contained in $\partial Du_n(0)$. This contradicts Lemma \ref{sandwich}, which says that $Du_n(0)\subset - (\R_{\geq 0}^n)$.

We conclude that all coordinates of $\nu$ are strictly positive.
\end{proof}

\begin{lem}\label{Gaussmap}
(Gauss map)
Let $n\geq 2$.
The natural projection $\mathbb{L}\to C_{std}$ lifts to a real analytic map $\mathbb{L}\to \tilde{C}_{std}$. For any $\mathfrak{y}\in \partial Du_n(0)$, this lifted map sends $(0, \mathfrak{y}) \in \mathbb{L} $ to the point in $\mathbb{RP}^{n-1}$ representing the line along the normal direction to $\partial Du_n(0)$ at $\mathfrak{y}$. 
\end{lem}

\begin{proof}
In Cor. \ref{graphicalnearvertexlem}, the parameter $(a_i)$ for the tangent space is up to scaling the same as the normal vector to $\partial Du_n(0)$. In particular $a_i>0$ for any $i=1,\ldots n$ by Lemma \ref{normalvectorhypersurface}, and we can set $a_n=1$.

In the notation of Cor. \ref{graphicalnearvertexlem}, 
\[
\tilde{x}_i= x_i- a_i x_n ,\quad i=1,\ldots n-1.
\]
By the graphical representation (\ref{graphicalnearvertex}), the ratios $\frac{ x_i}{x_n}$ extend to real analytic functions over the locus $x_n=0$, namely the part of $\mathbb{L}$ lying over the vertex at the origin. At the point $(0, \mathfrak{y})$, the tangent space information shows that the value of $\frac{ x_i}{x_n}$ is $a_i$. Since $a_i>0$, the extension lands inside the correct subset of $\mathbb{RP}^{n-1}$, so we obtain the desired lifting map $\mathbb{L}\to \tilde{C}_{std}$. 
\end{proof}

\begin{lem}
The lifted map is injective.
\end{lem}

\begin{proof}
It suffices to prove injectivity above the vertex $0\in \Delta_n$. Suppose the contrary, then two distinct points $\mathfrak{y}, \mathfrak{y}'\in \partial Du_n(0)$ share the same normal vector. Using the convexity of $Du_n(0)$, the line segment joining these two points is also contained in $\partial Du_n(0)$. Then by real analyticity, the entire line through the two points is contained in $\partial Du_n(0)$, contradicting again Lemma \ref{sandwich}.
\end{proof}

\begin{lem}
The lifted map is surjective.
\end{lem}

\begin{proof}
It suffices to show surjectivity over $0\in \Delta_n$. Let $\nu\in \R_{>0}^n$ represent a point in the open subset of $\mathbb{RP}^{n-1}$ inside $\tilde{C}_{std}$. Then $\langle \nu, \cdot\rangle$ defines a linear function on $Du_n(0)$. By Lemma \ref{sandwich}, the closed convex set $Du_n(0)$ is contained inside $-(\R_{\geq 0}^n)$, so the linear function must achieve an interior maximum at some boundary point $\mathfrak{y}$. Then $\nu$ is the normal vector at $\mathfrak{y}$, hence $(0, \mathfrak{y})\in \mathbb{L}$ maps to the corresponding point in $\tilde{C}_{std}$.
\end{proof}

\begin{lem}
The lifted map is proper.
\end{lem}

\begin{proof}
By the interior gradient estimate for $u_n$ on $\text{Int}(\Delta_n)$, it suffices to focus on the vertex region, and prove that the preimage of the bounded sets in $\tilde{C}_{std}$,
\begin{equation}\label{propernesssubset}
\Lambda^{-1} \leq \frac{x_i}{x_j} \leq \Lambda, \quad |x|\leq \delta\ll \Lambda^{-1}\ll 1, \quad \forall i,j,
\end{equation}
lies in a bounded region in $\mathbb{L}$.

By Cor. \ref{graphicalnearvertex} and Lemma \ref{normalvectorhypersurface}, at any point $(0,\mathfrak{y})\in \mathbb{L}$, the tangent space is of the form
\[
(x_1,\ldots x_n)\in \R\text{-span}(a_1,\ldots a_n) ,\quad \sum_1^n a_i(y_i-\mathfrak{y}_i)=0,\quad (a_i)\in \R_{> 0}^n,\quad \max_i a_i=1. 
\]
We will focus on the case where $a_n=\max_i a_i=1$, and use the chart (\ref{graphicalnearvertex}) from Cor. \ref{graphicalnearvertexlem}. (Otherwise we use some variant of the chart associated to $\max_i a_i$). By the $C^{k,\alpha}$-regularity of $\mathbb{L}$ (\cf Thm. \ref{asymptotic}), the real analytic functions $f_i$ in (\ref{graphicalnearvertex}) have uniform $C^1$ bounds on each coordinate chart, so the functions $\frac{x_i}{x_n}$ on the coordinate charts have a \emph{uniform Lipschitz constant} independent of the position of $\mathfrak{y}$.

From Thm. \ref{asymptotic}, we obtain that
any point $(x,y)\in \mathbb{L}$ in the preimage of $\{ |x|\leq \delta\} $ must lie within $C|x|=O(\delta)$ distance to some $(0, \mathfrak{y})\in \mathbb{L}$, so in particular lies on one of these charts, and without loss $a_n= \max_i a_i=1$. Suppose that $(x,y)$ lies in the preimage of (\ref{propernesssubset}), the uniform Lipschitz bound then gives $|\frac{x_i}{x_n}-\frac{a_i}{a_n}|\leq C\delta\ll \Lambda^{-1}$ for $i=1,\ldots n$, so
\[
(2\Lambda)^{-1} \leq a_i \leq 1, \quad i=1,\ldots, n.
\]
Notice $(a_1,\ldots a_n)$ is a normal vector to $Du_n(0)$ at the boundary point $\mathfrak{y}$, so
\[
\sum a_i (\mathfrak{y}_i- y_i')\geq 0,\quad \forall y'\in Du_n(0).
\]
Fixing a point $y_i'\in Du_n(0)$, and using $\mathfrak{y}_i\leq 0$ from Lemma \ref{sandwich}, we get
\[
(2\Lambda)^{-1}\sum \mathfrak{y}_i \geq \sum a_i \mathfrak{y}_i \geq \sum a_i  y_i' \geq \sum y_i' \geq -C,
\]
whence $\sum \mathfrak{y}_i \geq -C\Lambda$ and $\mathfrak{y}$ is bounded, and so must be  $(x,y)\in \mathbb{L}$.
\end{proof}

Thus the lifted map $\mathbb{L}\to \tilde{C}_{std}$ is a smooth and proper map which is also a bijection, and therefore a homeomorphism. We conclude Thm. \ref{Topology}.

\begin{rmk}
If we know that $\partial Du_n(0)$ is a strictly convex hypersurface, \ie the principal curvatures are \emph{strictly positive} everywhere, then the lifted map $\mathbb{L}\to \tilde{C}_{std}$ is a diffeomorphism. 
\end{rmk}

\subsection{Appendix: Strong maximum principle}

Lemma \ref{strongmax} is a version of strong maximum principle for possibly singular minimal surfaces in any codimension, which must be well known, but we supply a proof for convenience.

\begin{lem}\label{strongmax}
(Strong maximum principle) Let $n\geq 2$. Let $X$ be an $n$-dimensional stationary integral varifold  inside $B_1\subset \R^N$, such that support contains the origin, and lies inside the half space $x_1\geq 0$. Then 
\[
\text{Mass}( B_1\cap \{ x_1=0\} ) \geq \epsilon_0(n), 
\]
for some universal constant $\epsilon_0$ depending only on $n$.

\end{lem}

\begin{proof}
Let $\chi_k(|x|)$ be a sequence of radial cutoff function on $\R^N$ such that
\[
\begin{cases}
0\leq \chi_k\leq 2^{-k}, \quad |D\chi_k|\leq C,\quad |D^2\chi_k|\leq C2^{k}.
\\
\chi_k=0,\quad |x|\geq  \frac{1}{2^k},
\\
\chi_k=2^{-k},\quad |x|\leq \frac{1}{2^{k+1}}.
\end{cases}
\]
The gradient of $\chi_k$ is supported on the annulus region $\{ 2^{-k-1}\leq  |x| \leq 2^{-k} \}$.
Let $\epsilon>0$ be a small parameter, and consider the function
\[
f= - x_1+  \epsilon \chi_k(x),\quad f^+=\max(f,0).
\]
Since $x_1\geq 0$ on the support of $X$, we know $f\leq 0$ near the boundary of $B_1$. By choosing $\epsilon$ generic, we can ensure that $\{ f=0\}\cap \text{supp}(X)$ has zero $\mathcal{H}^n$-measure. By the stationarity of $X$ and a standard approximation argument by $C^1$ test functions, 
\[
0=\int_X  div_{T_xX} (f^+ \nabla f) d\norm{X}=  \int_{f>0}  |\nabla f|^2 d\norm{X} + \int_{f>0} f \Tr_{T_xX} \text{Hess}(f) d\norm{X},
\]
so
\[
\int_{f>0}  |\nabla f|^2 d\norm{X} \leq  C 2^{k}\epsilon \int_{\{ f>0\}   } f\norm{X} .
\]
Thus by the H\"older inequality,
\[
\begin{split}
& (\int_{f>0}  |\nabla f| d\norm{X})^2
\\
\leq & (\int_{f>0}  |\nabla f|^2 d\norm{X} ) \text{Mass}(\{f>0\})
\\
\leq &  C 2^{k}\epsilon (\int_{\{ f>0\}   } f\norm{X} ) \text{Mass}(\{f>0\})
\\
 \leq & C2^k \epsilon \text{Mass}(\{f>0\})^{ 1+ \frac{1}{n} } (\int_X |f^+|^{\frac{n}{n-1}}  d\norm{X} )^{ \frac{n-1}{n} }. 
\end{split}
\]
But by the Michael-Simon-Sobolev inequality \cite{MichaelSimon} for stationary varifolds,
\[
(\int_X |f^+|^{\frac{n}{n-1}}  d\norm{X} )^{ \frac{n-1}{n} } \leq C\int_X  |\nabla f^+| d\norm{X}.
\]
Hence
\[
(\int_X |f^+|^{\frac{n}{n-1}}  d\norm{X} )^{ \frac{n-1}{n} } \leq C2^k \epsilon \text{Mass}(\{f>0\})^{ 1+ \frac{1}{n} } . 
\]
From the support information on $f^+$, we deduce
\[
\begin{split}
& \epsilon 2^{-k-1} \text{Mass}(\{ x_1\leq \epsilon 2^{-k-1}\}\cap B_{2^{-k-1}})^{  \frac{n-1}{n} } \leq (\int_X |f^+|^{\frac{n}{n-1}}  d\norm{X} )^{ \frac{n-1}{n} } 
\\
 \leq & C2^k \epsilon \text{Mass}(\{f>0\})^{ 1+ \frac{1}{n} } \leq C2^k \epsilon \text{Mass}(\{  x_1\leq \epsilon 2^{-k}  \}\cap B_{2^{-k}})^{ 1+ \frac{1}{n} }. 
\end{split}
\]

Let $a_k=\text{Mass}(\{  x_1\leq \epsilon 2^{-k}  \}\cap B_{2^{-k}}) $. We have proven
\[
a_{k+1} \leq C2^{2kn/(n-1)} a_k^{ 1+ \frac{2}{n-1} }. 
\]
Take $p= 1+\frac{1}{n-1}$, and suppose $a_0<\epsilon_0(n)$ is sufficiently small depending on $n$. 
Then we can inductively deduce the double exponential decay estimate
$
a_k\leq a_0^{p^k}. 
$
On the other hand, by assumption the origin lies on the support of $X$, so by the monotonicity formula for stationary varifolds,
\[
a_k\geq \text{Mass}(B(\epsilon 2^{-k})) \geq \omega_n(\epsilon 2^{-k})^n.
\]
This is incompatible with the double exponential decay for $k\to +\infty$. This contradiction implies that $a_0\geq \epsilon_0(n)$. Since this estimate is independent of the small $\epsilon$, we can take the $\epsilon\to 0$ limit to deduce
\[
\text{Mass}(\{  x_1\leq 0  \}\cap B_{1}) \geq \epsilon_0(n)
\]
as required.
\end{proof}


\begin{thebibliography}{7}

\bibitem{Allard}
 William K. Allard. On the First Variation of a Varifold. The Annals of Mathematics, 95(3):417,
1972



\bibitem{Caff} 
Caffarelli, L.; Nirenberg, L.; Spruck, J. The Dirichlet problem for nonlinear second-order elliptic equations. III. Functions of the eigenvalues of the Hessian. Acta Math. 155 (1985), no. 3-4, 261--301.






\bibitem{HarveyLawson1}
Harvey, Reese; Lawson, H. Blaine, Jr. Calibrated geometries. Acta Math. 148 (1982), 47--157.





\bibitem{HarveyLawson} 
Harvey, F. Reese; Lawson, H. Blaine, Jr. Dirichlet duality and the nonlinear Dirichlet problem. Comm. Pure Appl. Math. 62 (2009), no. 3, 396--443.



\bibitem{JoyceU1} 
Joyce, Dominic. $\rm U(1)$-invariant special Lagrangian 3-folds. III. Properties of singular solutions. Adv. Math. 192 (2005), no. 1, 135--182. 



  
\bibitem{Matessi} 
Matessi, Diego. Lagrangian pairs of pants. Int. Math. Res. Not. IMRN 2021, no. 15, 11306--11356.



\bibitem{Matessi2} Matessi, Diego. Lagrangian submanifolds from tropical hypersurfaces. Internat. J. Math. 32 (2021), no. 7, Paper No. 2150046, 63 pp.

\bibitem{MichaelSimon}

Michael, J. H.; Simon, L. M. Sobolev and mean-value inequalities on generalized submanifolds of $R\sp{n}$. Comm. Pure Appl. Math. 26 (1973), 361--379.

  
\bibitem{Mikhalkin}
Mikhalkin, Grigory. Mikhalkin, Grigory. Decomposition into pairs-of-pants for complex algebraic hypersurfaces. Topology 43 (2004), no. 5, 1035--1065.







  
\bibitem{Mikhalkin2} 
Mikhalkin, Grigory. Examples of tropical-to-Lagrangian correspondence. Eur. J. Math. 5 (2019), no. 3, 1033--1066.


	
		
\bibitem{MooneySavin}	 
Mooney, Connor; Savin, Ovidiu. Non $C^1$ solutions to the special Lagrangian equation.	arXiv:2303.14282.
		

\bibitem{Savin}
Savin, Ovidiu. Small perturbation solutions for elliptic equations. Comm. Partial Differential Equations 32 (2007), no. 4-6, 557--578.



 
\bibitem{Yuan}	
Chen, Jingyi; Warren, Micah; Yuan, Yu. A priori estimate for convex solutions to special Lagrangian equations and its application. Comm. Pure Appl. Math. 62 (2009), no. 4, 583--595.

\bibitem{Yuan2} 
Wang, Dake; Yuan, Yu. Hessian estimates for special Lagrangian equations with critical and supercritical phases in general dimensions. Amer. J. Math. 136 (2014), no. 2, 481--499.		
		
		
		
		
	
		
		
	

		
		

		

	
		
		
	\end{thebibliography}
\end{document}